\documentclass[reqno]{amsart}
\usepackage{amssymb}
\usepackage[dvips]{epsfig}
\usepackage{graphicx}
\usepackage{color}
\usepackage{amsmath}
\usepackage{amssymb}
\usepackage{amsfonts}
\usepackage{amsthm}
\usepackage{mathrsfs}

\newtheorem{thm}{Theorem}[section]
\newtheorem{lem}[thm]{Lemma}

\newtheorem{cor}[thm]{Corollary}
\theoremstyle{remark}
\newtheorem{rem}{\bf Remark}[section]
\theoremstyle{definition}
\newtheorem{defn}[thm]{Definition}

\numberwithin{equation}{section}
\usepackage[colorlinks=true,pdfstartview=FitV,linkcolor=magenta,citecolor=cyan]{hyperref}

\usepackage{bm}
\begin{document}

	\title[ ]{how the angle of static field affect the spectrum of lattice schr\"odinger operator on $\mathbb{Z}^2$ }
	
	\author[Meina Gao]{Meina Gao}
	\address{School of Mathematics, Phyiscs and Statistics, Shanghai Polytechnic University, shanghai, P.R.China}
	\email{mngao@sspu.edu.cn}

\author[Siming Li]{Siming Li}
	\address{School of Mathematical Sciences, Dalian University of Technology, Dalian, P.R.China}
	\email{lisiming@mail.dlut.edu.cn}
	
	\author[Yingte Sun]{Yingte Sun}
	\address{School of Mathematics, Yangzhou University, yangzhou, P.R.China} \email{sytsts@aliyun.com}

	\date{}



\begin{abstract}The spectral properties of the discrete Schrödinger operator on a two‑dimensional lattice under uniform static fields are critically governed by the field direction: for angles parametrized by irrational frequencies, the operator is exponentially localized, and for a subset of such angles with asymptotically full Lebesgue measure, this localization persists under a given logarithmically decaying perturbation potential.

\end{abstract}
\maketitle

\section{Introduction}

In this paper, we investigate the lattice Schr\"odinger operator on the two-dimensional  lattice $\mathbb{Z}^2$ subject to uniform static  fields. Let $$H_{\omega}=\Delta+D : \quad\ell^{2}(\mathbb{Z}^2)\rightarrow \ell^{2}(\mathbb{Z}^2)$$
be defined by
\begin{equation}\label{25.8.18.1}
	(H_{\omega}u)(n)=(\Delta u)(n)+(Du)(n),\quad u=\{u(n)\}_{n\in\mathbb{Z}^2}\in\ell^2(\mathbb{Z}^2),
\end{equation}
where $\omega=(\omega_1,\omega_2)\in\mathbb{R}^2$, and
$n=(n_1,n_2)\in\mathbb{Z}^2$. The discrete Laplacian $\Delta$ and the multiplication operator $D$ are respectively defined  by
\begin{equation}
	(\Delta u)(n)=\sum_{m\in \mathbb{Z}^2, |n-m|=1}u(m)
\end{equation}
and
\begin{equation}\label{25.8.31.3}
	(Du)(n)=\left( \omega\cdot n\right)  u(n),
\end{equation}
where
$|n|:=|n_1|+|n_2|$ 
and
$\omega\cdot n:=\omega_1 n_1+\omega_2 n_2.
$ Here, the multiplication operator $D$ represents a static field oriented along the direction $(\omega_1,\omega_2)$, with  field strength $\sqrt{\omega^2_1+\omega^2_2}$.

This paper aims to investigate the spectral properties of Schr\"odinger operators on two-dimensional lattices in the presence of static fields, with particular emphasis on the role played by the spatial orientation of the field. Our results demonstrate  that the direction of static fields has a crucial influence on the spectral properties of the operator.

In the presence of  disordered potentials, the Schr\"odinger operator on a lattice typically exhibits purely point spectrum, with corresponding eigenstates decay exponentially in space. This phenomenon is known as Anderson localization. Physicists generally conjecture that  disordered potentials of any strength can induce localization for one- and two-dimensional Schr\"odinger operators. In contrast, for dimensions  three and higher, localization occurs only in the presence of  sufficiently strong disorder; under weak disorder, the operator is expected to retain  continuous spectrum, implying a spectral phase transition as the strength of disorder varies.   Over the past four decades, many mathematicians have made fundamental contributions to this field. These include, among others, the works of Fröhlich–Spencer \cite{FS1983}, Klein \cite{DA1989,GK2001}, Aizenman \cite{AM1993}, and Bourgain \cite{BK2005}. Nevertheless, two important problems remain unresolved to this day: first, whether localization can be rigorously established for two-dimensional lattices under arbitrarily weak disorder; and second, whether absolutely continuous spectrum exists in high-dimensional systems with weak disorder.

Unlike disorder-induced localization, Wannier–Stark localization originates from the Wannier–Stark ladder generated by  external static fields. In such a setting, particles living on the one-diemnsional lattice become super-exponentially localized under uniform static fields. To date, both physical experiments and mathematical theories have provided a comprehensive understanding of this phenomenon. From the mathematical perspective, recent works by Sun–Wang \cite{SW2024,SW2025} and Aloisio \cite{M2025,ADM2026} have demonstrated that this localization remains stable under long-range hopping or arbitrary bounded perturbations. De Roeck et al. \cite{DHLV2026} further extended the analysis to interacting few-body systems, establishing that for finitely many interacting particles on a lattice subject uniform static fields, the spectrum is purely point and the eigenstates decay exponentially in certain spectral regions.

However, for higher-dimensional lattice systems,  the question of whether a static field can still generate a Wannier–Stark ladder and thereby induce localization has not yet been fully resolved in the physics community. Some physical studies \cite{TTM1993,TTM1999,KB2013,KK2002} suggest that the direction of  static fields may significantly influence the spectral structure of the operator. Nevertheless, most of these investigations are limited to finite-volume models or rely on numerical simulations, thus lacking rigorous mathematical justification.  Motivated by this gap, the present paper systematically examines relevant physical models on the two-dimensional lattice from a rigorous mathematical perspective, focusing  on the following two interrelated aspects:

$\bullet$ \textbf{Spectral classification depending on field direction}

By exploiting the T$\ddot{\mbox{o}}$plitz structure of the operator introduced in recent works \cite{SW2024,HS2025}, we provide a complete analysis of how the field direction affects the spectral properties of the Schrödinger operator. Our analysis also indicates that the nature of the hopping terms, whether  nearest-neighbor, short-range, or long-range, may play an important role in determining distinct spectral classifications.

$\bullet$ \textbf{Persistence of localization under perturbations }

As noted above, in the one-dimensional case, arbitrary bounded perturbations do not destroy Wannier–Stark localization. We thus investigate whether this localization persists under perturbations in the two-dimensional setting. To this end, we consider the following Schrödinger operator under static fields:
\begin{equation}\label{jiaodu}
	(\mathcal{H}_{\theta,\varepsilon}u)(n)=(\Delta u)(n)+(\omega\cdot n)u(n)+\varepsilon v(n) u(n),
\end{equation}
where $\omega=(\omega_1,\omega_2)=(\cos \theta,\sin \theta)$, so that the field  strength  satifies $\sqrt{\omega_1^2+\omega_2^2}=1$.

Our result asserts that,  for most  field directions $\theta$,  exponential localization persists even in the presence of a  perturbation potental $v$ that decays logarithmically in the lattice disantce. The proof of this persistence  relies on KAM-type diagonalization techniques for certain quantum oscillators \cite{BG2001,LY2010}, in which the direction angle $\theta$ is introduced as a KAM iteration parameter.

The  paper is organized as follows. In Section 2, we present the two main results of this paper: the spectral classification depending on the field direction, and the persistence of localization under perturbations with logarithmic decay. In Section 3, we provide a preliminary introduction to operator norms and the relevant  function spaces. In Section 4, we give a rigorous proof of how the direction of  static fields affects the spectral properties of the operator. In Section 5, we further refine the setting of operator norms and, by employing KAM diagonalization, prove that for a fixed logarithmically  decaying perturbation potential, the localization property  is preserved for the static field in most directions.

%

\section{Main results}

In this section, we present the two main results of this paper.

\subsection{Spectral classification depending on field direction}

\begin{thm}
	Let \(H_\omega\) be the lattice Schrödinger operator defined in \eqref{25.8.18.1} with \(\omega\in\mathbb{R}^2\setminus\{0\}\). The spectral properties of \(H_\omega\) are classified according to the direction of \(\omega\) as follows:
	
	\begin{itemize}
		\item[(I)] \textbf{Axis-parallel direction:} If \(\omega_1=0\) or \(\omega_2=0\), then the spectrum \(\sigma(H_\omega)\) is purely absolutely continuous.
		
		\item[(II)] \textbf{Non-axis-parallel rational direction:} If there exists a  vector \(\ell\in\mathbb{Z}^2\setminus\{0\}\) such that \(\omega\cdot\ell=0\), but \(\omega_1\neq0\) and \(\omega_2\neq0\), then \(\sigma(H_\omega)\) is purely essential.
		
		\item[(III)] \textbf{Irrational direction:} If \(\omega\cdot\ell\neq0\) for every  \(\ell\in\mathbb{Z}^2\setminus\{0\}\), then \(\sigma(H_\omega)\) is purely point. Moreover, for every \(a>0\), there exists a complete orthonormal basis \(\{\psi_n\}_{n\in\mathbb{Z}^2}\) of eigenvectors of \(H_\omega\) such that, for each \(n\in\mathbb{Z}^2\) and all \(m\in\mathbb{Z}^2\),
		\[
		|\psi_n(m)| \le C_{\omega,a} \, e^{-a|m-n|},
		\]
		where \(|m-n|=|m_1-n_1|+|m_2-n_2|\) and \(C_{\omega,a}\) is a positive constant depending only on \(\omega\) and \(a\).
	\end{itemize}
\end{thm}

The theorem provides a complete characterization of the spectral properties of the Schr\"odinger operators on the two‑dimensional lattice in the presence of  static fields, depending on the field direction. This characterization relies critically on the fact that the lattice operator involves only nearest‑neighbour hopping, a case known in physics as the tight‑binding model. Consequently, in Case I the two‑dimensional operator can be expressed as the tensor product of two one‑dimensional operators; this structure permits a spectral measure analysis, which shows that the spectrum is purely absolutely continuous. In Cases II and III, for every direction distinct from that of Case I, a unitary transformation can be constructed to diagonalize the operator. However, if the nearest‑neighbour hopping is replaced by other short‑range hopping (and a fortiori by long‑range hopping), there exist many directions for which no such unitary diagonalization is possible. Thus the spectral properties are also sensitive to the range of the hopping. To our knowledge, this dependence has not been observed in previous studies of (quasi‑periodic or random) Schr\"odinger operators.

We should also note that in Case II, the spectrum is purely essential. More specifically, it consists of eigenvalues of infinite multiplicity: each eigenvalue admits infinitely many linearly independent eigenvectors. Moreover, these eigenstates are shown to decay exponentially, with localization centers tending to infinity. Consequently, extended states can be constructed by summing all of these eigenstates that lie in $\ell^{\infty}(\mathbb{Z}^2)$
but not in $\ell^2(\mathbb{Z}^2)$.

\subsection{Persistence of localization under perturbations }\

\begin{defn}(Diophantine angle)\label{Diophantine} 
	We say that $\theta \in [0,2\pi]:=\mathcal{O}$ is a Diophantine angle, if there exist $\gamma>0$ and $\tau>1$ such that $\theta \in \mathrm{DC}_{\gamma,\tau}$, where
$$\mathrm{DC}_{\gamma,\tau}:=\Big\{\theta \in \mathcal{O}:|n_1\cos\theta+n_2\sin\theta|\geq \frac{\gamma}{|n|^{\tau}}, \quad \forall n=(n_1,n_2) \in \mathbb{Z}^2 \backslash \{0\} \Big\}.$$
\end{defn}
\begin{thm}\label{thm2}
	Let $0<\gamma\ll 1$, $\tau>1$, and  $\theta \in \mathrm{DC}_{\gamma,\tau}$. Consider the Schr\"odinger operator $\mathcal{H}_{\theta,\varepsilon}$ defined in \eqref{jiaodu}, which acts on $\ell^2(\mathbb{Z}^2)$ and includes static fields of unit strength with direction parameter $\theta$.  Fix constants $0<\mu<1$ and $\rho,a>0$. Assume that the perturbation $v:\mathbb{Z}^2\to\mathbb{R}$ satisfies
	$$
	|v(n)| \le \bigl(\ln \langle n\rangle\bigr)^{-\sigma}, \qquad \forall n\in\mathbb{Z}^2,
	$$
	where $\langle n\rangle:=\max \{e, |n|\}$ and $\sigma > \tau/\mu$.
	
	Then there exists a sufficiently small constant $\varepsilon_0 = \varepsilon_0(\mu,\tau,\sigma,\gamma,\rho)>0$ such that for every $0<\varepsilon<\varepsilon_0$, there exists a Cantor-like set $\mathcal{O}_{\varepsilon} \subset \mathrm{DC}_{\gamma,\tau}$ satisfying
	\[
	\mathrm{meas}\bigl(\mathrm{DC}_{\gamma,\tau}\setminus \mathcal{O}_{\varepsilon}\bigr) \le C_0\,\varepsilon^{1/10},
	\]
	where $C_0 = C_0(\mu,\tau,\sigma,\gamma,\rho)>0$ is a constant. For every $\theta \in \mathcal{O}_{\varepsilon}$, the operator $\mathcal{H}_{\theta,\varepsilon}$ has purely point spectrum, and every eigenvalue is simple. Moreover, there exists a complete orthonormal basis $\{\psi_n\}_{n\in\mathbb{Z}^2}$ of eigenfunctions such that for all $m,n\in\mathbb{Z}^2$,
	\[
	|\psi_n(m)| \le C_1 \, e^{-\frac{a}{4}|m-n|},
	\]
	with $C_1 = C_1(\mu,\tau,\sigma,\gamma,\rho,\varepsilon)>0$.
\end{thm}
\begin{rem}\
	
	$\bullet$ From Lemma \ref{Mes3}, one has the measure estimation : $$\text{meas}\big([0,2\pi]\setminus \mathrm{DC}_{\gamma,\tau}\big) \leq \widetilde{C}\gamma.$$
	
	$\bullet$ Under our assumptions on the parameters $\tau$ and $\mu$, the perturbation term $v$ decays like a power of a logarithm (i.e.,$|v(n)| \le \bigl(\ln \langle n\rangle\bigr)^{-\sigma}, \sigma>1$),  rather than merely like the inverse logarithm.
	
	$\bullet$ The above theorem can be directly generalized to the case of exponentially decaying long-range hopping (satisfying the T\"oplitz property), and the proof remains essentially unchanged. Moreover, the argument can be extended to $\mathbb{Z}^d$ with $d\geq3$. In that setting, one introduces angles $(\theta_1,\cdots,\theta_{d-1})\in[0,2\pi]^{d-1}$ in polar coordinates as parameters, and it suffices to invoke Lemma A.6 in \cite{SW24} to handle the degeneracy issues arising in the  measure estimates. We restrict our attention to $\mathbb{Z}^2$ primarily to better illustrate the influence and role of the direction of the static field.
	
	$\bullet$ More precisely, the above theorem states that,  for a fixed small perturbation with power of logarithmic decay, the localization properties are preserved for most operators in the family parameterized by the direction of the static field. A similar phenomenon was studied for one‑dimensional quasi‑periodic Schr\"odinger operators in \cite{DLYZ2023}; there, however, for a fixed small perturbation with exponential decay, the localization properties are preserved for almost all operators in the quasi‑periodic family.

\end{rem}

Indeed, having established the uniform exponential localization of the eigenstates of the operator, it is straightforward to derive the corresponding exponential dynamical localization.
\begin{cor}
	Under the same assumptions of Theorem \ref{thm2}, for any $\theta \in \mathcal{O}_{\varepsilon}$,  the Schr\"odinger operator $\mathcal{H}_{\theta, \varepsilon}$ shows exponential dynamical localization: there exists some $C_2$ depending on $\mu,\tau,\sigma,\gamma,\rho,\varepsilon$, such that for any $m,n \in \mathbb{Z}^2$:
\begin{equation}
	\sup_{t\in\mathbb{R}}\big| \big \langle e^{-\mathrm{i}t\mathcal{H}_{\theta, \varepsilon}}\delta_m, \delta_n \big \rangle\big|\leq C_2e^{-\frac{a}{4}|m-n|}.
\end{equation}
\end{cor}

We have established that the eigenstates of the Schr\"odinger operator in static fields are uniformly exponentially localized, and, importantly, this uniformity persists under perturbations with logarithmic decay. This result naturally raises two questions concerning the relationship between perturbation decay and the localization type.

$\bullet$ First, whereas non‑uniform localization is generic in most quasi‑periodic settings, it remains unknown whether, in the static‑field context, even slower‑decaying perturbations or non‑decaying i.i.d. random perturbations can induce non‑uniform localization. Resolving this would clarify how the decay rate of the perturbation qualitatively determines the type of localization (uniform vs. non‑uniform).  Classical KAM techniques are largely limited to proving uniform localization; non‑uniform  diagonalization frameworks have been developed by Eliasson \cite{E1997} for quasi‑periodic operators and Imbrie \cite{J2016} for random ones, but their applicability to the present problem is still open.

$\bullet$  Second, monotonic quasi‑periodic operators are also known to possess uniform exponential localization \cite{CSZ25,KPS25}. A parallel question is whether this uniform localization remains stable under slowly decaying (non‑exponential) perturbations, such as power‑law or logarithmic, in the same spirit as our analysis for the static‑field case.

\section{Functional setting}

Let $a\geq0$, the subspace $\mathcal{S}:=\ell^1_{a}(\mathbb{Z}^2)\subset\ell^1(\mathbb{Z}^2)$
consist of all sequences
\begin{equation}\label{25.9.7.1}
	u=\{u(n)\}_{n\in\mathbb{Z}^2}\in\mathbb{C}^{\mathbb{Z}^2}
	\end{equation}
with finite norm
\begin{equation}
	\left\| u\right\|_{a}=
	\sum_{n\in\mathbb{Z}^2}|u(n)| e^{a| n |}.
	\end{equation}
Similarly, the space $\ell^{\infty}(\mathbb{Z}^2)$
consists of all $u$ defined in \eqref{25.9.7.1} with finite norm
\begin{equation*}|u|_{\infty}=\sup_{n\in\mathbb{Z}^2}|u(n)|.\end{equation*}

Consider a linear operator $A$ on
$\mathcal{S}$
with the matrix representation $A=(A_{i,j})_{i,j\in\mathbb{Z}^2}$ given by
\begin{equation}\label{25.8.30.1}
	A_{i,j}=\langle \delta_{i},A\delta_{j}\rangle
	\end{equation}
with
$\langle \cdot, \cdot\rangle$ denoting the canonical pairing between $\ell^\infty$ and $\ell^1$, and $\{\delta_i\}_{i\in\mathbb{Z}^2}$ being the standard basis of $\ell^1(\mathbb{Z}^2)$.
Furthermore, we denote $[A]$ as the diagonal part of $A$, i.e.,
\[
[A]_{i,j} :=
\begin{cases}
	A_{i,j}, & i = j, \\
	0, & \text{otherwise.}
\end{cases}
\]

\begin{defn}\label{space s}
	Let $a\ge 0$. We say that a linear operator $A$ on $\mathcal{S}$ belongs to the operator space
	$\mathcal{B}_{a}$ if the following norm is finite:
	\begin{equation}\label{norm1}
		\left\| A\right\|_{a}
		:=  \sum_{n\in\mathbb{Z}^2} e^{a| n |}
		\sup_{i-j=n}
		|A_{i,j}|<+\infty.
	\end{equation}
Suppose the operator $A$ is Lipschitz continuous with respect to the parameter $\theta\in\mathcal{O}$. For any $\alpha>0$, we say that $A(\theta)\in \mathcal{B}_{a}^{\alpha} $ if
\begin{equation}\label{norm_lip}
	\|A\|_a^\alpha := \sup_{\theta \in \mathcal{O}}  \|A\|_a + \sup_{\theta_1\neq\theta_2 \in \mathcal{O}} \alpha
	\frac{ \left\| \bigtriangleup A\right\| _a}{\left| \bigtriangleup \theta \right| } < +\infty,
\end{equation}
where we denote $\bigtriangleup A :=A(\theta_1) - A(\theta_2)$
and $\bigtriangleup \theta :=\theta_1 - \theta_2 $.

\begin{rem}It is readily verified that if an operator
	$A$ belongs to
	$\mathcal{B}_{a}$, then it is also a bounded operator on the standard
	$\ell^2$ space. Consequently, self-adjoint and unitary operators can be defined with respect to the canonical inner product on $\ell^2$.
\end{rem}

\end{defn}

\begin{defn}(The commutator).
For two operators $A,B\in\mathcal{B}_a$, define the commutator as
\begin{equation*}
	[A,B]:=AB-BA.
\end{equation*}
The adjoint action of $A$ on $B$ is defined recursively as$$\text{ad}_A(B) := [A, B],$$and for $k \ge 2$,$$\text{ad}_A^k(B) := [A, \text{ad}_A^{k-1}(B)].$$We refer to $\text{ad}_A^k(B)$ as the $k$-th order commutator for $k \ge 1$.
\end{defn}

%

\begin{lem}\label{lem8.28.1}
	Let $a>0$, $A,B\in\mathcal{B}_a$ and $u\in\mathcal{S}$. Then, the following assertions hold:
	
	\begin{enumerate}
		\item  The vector $Au\in\mathcal{S}$ satisfies
		\[
		\left\| Au\right\|_a\leq
		\left\| A \right\|_a \left\| u\right\|_a.
		\]
		\item The operators $AB, A^n\in \mathcal{B}_a$ satisfy respectively
		\[	\left\| AB \right\|_a \leq \left\| A \right\|_a\left\| B \right\|_a
		\quad\text{and}\quad
			\left\| A^n \right\|_a \leq
		\left\| A \right\|_a^n.
		\]
		\item For any integer $k\ge1$, the $k$-th order commutator
		$\mathbf{ad}_A^k(B) \in \mathcal{B}_a$ satisfies
      	\[		\left\| \mathbf{ad}_A^k(B)\right\| _a
      	\leq(2\left\| A \right\|_a)^k\left\| B \right\|_a.\]
		\item Let $\phi=e^{A}$ with $A\in\mathcal{B}_a$. If $\|A\|_a \le 1$, then:$$\|\phi^{\pm 1}-\mbox{Id}\|_a \le e^{\|A\|_a} - 1 \le 2\|A\|_a.$$
		
		\item Given $N \in \mathbb{N}$, we can define the truncation operator $\Pi_{N}$ for operator $A$,
		\begin{align}
			\Pi_{N}A_{n,n'}=\begin{cases}
				A_{n,n'},\ \ |n-n'|\leq N,\\
				0,\quad\quad\   \mbox{otherwise}.
			\end{cases}
		\end{align}
		For any $0<\delta<a$, the operator $\Pi_{N}^{\bot}=Id-\Pi_{N}$ satifies 
		\begin{equation}
			\|\Pi_{N}^{\bot}A\|_{a-\delta} \leq e^{-\delta N} \|A\|_{a} .
			\end{equation}
		
		\item Items (1)-(4) still hold when replacing the norm $\|\cdot\|_{a}$ with $\|\cdot\|^{\alpha}_{a}$.
	\end{enumerate}
\end{lem}
\begin{proof}
	The proof of items (1)–(2) can be found in Lemma 2.4 of \cite{HS26}. As for items (3)–(4), they follow directly from the application and calculation based on items (1)–(2). Item (5) follows directly from Definition \ref{space s}.
\end{proof}

%

\begin{defn}(T\"{o}plitz operator).
   For the linear operator $A$ on $\mathcal{S}$ with the matrix representation \eqref{25.8.30.1}, if the matrix entries satisfy
	$$A_{i,j}=A_{i',j'},$$
	whenever
	$i-j=i'-j'$ for all $i,j,i',j'\in\mathbb{Z}^2$, then $A$ is called a T\"{o}plitz operator.
\end{defn}

For T\"{o}plitz operators, we have the following property.
\begin{lem}\label{lem25.8.31.1}
	The commutator of any two T\"{o}plitz operators vanishes.
\end{lem}

\begin{proof}Let $A, B$ be two T\"{o}plitz operators.
By Definition 3.4, their matrix entries satisfy $A_{i,k} = A_{i-k, 0}$ and $B_{k,j} = B_{k-j, 0}$ for all $i, j, k \in \mathbb{Z}^2$. For any $i, j \in \mathbb{Z}^2$, the entries of the product $AB$ are given by
$$(AB)_{i,j} = \sum_{k \in \mathbb{Z}^2} A_{i,k} B_{k,j}.$$Let $m = i + j - k$. As $k$ ranges over $\mathbb{Z}^2$, $m$ also ranges over $\mathbb{Z}^2$. By the T\"{o}plitz property, we have:
$$A_{i,k} = A_{i-k, 0} = A_{m-j, 0} = A_{m,j}, \quad B_{k,j} = B_{k-j, 0} = B_{i-m, 0} = B_{i,m}.$$
Substituting these into the summation, we obtain:
\begin{equation}\label{26.3.1.1}(AB)_{i,j} = \sum_{m \in \mathbb{Z}^2} A_{m,j} B_{i,m} = \sum_{m \in \mathbb{Z}^2} B_{i,m} A_{m,j} = (BA)_{i,j}, \end{equation}
where \eqref{26.3.1.1} is a direct consequence of the translation invariance of $A$ and $B$. It then follows that
$$[A, B] = AB - BA = 0.$$The proof is complete.
\end{proof}

\section{Spectral classification depending on field direction}

\subsection{Directions parallel to the lattice: Proof of Case I}\

Without loss of generality, we may assume that $\omega_2=0$. Then we will prove it under the following steps.

\textbf{Step 1: Tensor product decomposition.} Identify $\ell^2(\mathbb{Z}^2)$ with the Hilbert space tensor product $\ell^2(\mathbb{Z})\otimes\ell^2(\mathbb{Z})$ via
$$\ell^2(\mathbb{Z}^2)\ni\{\psi_{n,n}\}_{n,m\in\mathbb{Z}}\longleftrightarrow
\sum_{n,m\in\mathbb{Z}}\psi_{n,m}(\delta_n\otimes \delta_m),$$
where $\{e_n\}_{n\in\mathbb{Z}}$ is the standard orthonormal basis of $\ell^2(\mathbb{Z})$. Define two operators:
\begin{itemize}
\item $A_1$ on $\ell^2(\mathbb{Z})$ with domain
$$\mathcal{D}(A_1)=\{\chi\in\ell^2(\mathbb{Z}):\sum_{n\in\mathbb{Z}}|n\chi_n|^2<\infty\}$$
acting as
\begin{equation}\label{26.5.8.1} (A_1 \chi)_n = (\chi_{n-1}+2\chi_n+\chi_{n+1}) + \omega_1 n \chi_n. \end{equation}

\item \( A_2 \) on \( \ell^2(\mathbb{Z}) \) acting as  
\[ (A_2 \phi)_m =\phi_{m-1}+2\phi_m+\phi_{m+1}. \]
\end{itemize}
Then $A_2$ is bounded and self-adjoint, with spectrum 
$[-2,2]$ and purely absolutely continuous spectral type.
The operator $A_1$
is essentially self-adjoint (see Lemma \ref{lem26.5.8} in  Appendix) and its closure is self-adjoint. In the tensor product representation,
$$H=A_1\otimes I+I\otimes A_2,$$
understood as the closure of the sum defined on algebraic tensors.

\textbf{Step 2: Verification of the separability condition.} Fox's theorem in \cite{F1976} states that if there exist cores
$D_1\subset\mathcal{D}(A_1)$ and $D_2\subset\mathcal{D}(A_2)$ such that the linear span of $D_1\otimes D_2$ is a core for $H$, then for all $\psi_1\otimes\psi_2\in D_1\otimes D_2$,
$$H(\psi_1\otimes\psi_2)=(A_1\psi_1)\otimes\psi_2+\psi_1\otimes(A_2\psi_2),$$
and the spectral measures satisfy a convolution formula.

Let $$D_1=\{\chi\in\ell^2(\mathbb{Z}):\chi\ \mbox{has finite support}\}$$ and $$D_2=\{\phi\in\ell^2(\mathbb{Z}):\phi\ \mbox{has finite support}\}.$$
Then we can verify the following:
\begin{itemize}
	\item $D_1\subset\mathcal{D}(A_1):$ If $\chi$ has finite support, then
	$$(A_1\chi)_n=(\chi_{n-1}+2\chi_n+\chi_{n+1})+\omega_1n\chi_n$$
	also has finite support (the support expands by at most one lattice site in each direction). Hence $A_1\chi\in\ell^2(\mathbb{Z})$.
	\item $D_1$ is a core for $A_1$: Since $A_1$
	is essentially self-adjoint, any dense subspace that is invariant under the resolvent (or simply any dense subspace contained in $\mathcal{D}(A_1)$) is a core. Finite-support sequences are dense in $\ell^2(\mathbb{Z})$, and they are contained in $\mathcal{D}(A_1)$.
	\item $D_2$ is a core for $A_2$: The same argument applies because $A_2$ is bounded (in fact, $D_2$ is dense and contained in $\mathcal{D}(A_2)$).
	\item Core property for $H$: The algebraic tensor product $D_1\otimes D_2$ spans the set of all finite-support sequences on $\ell^2(\mathbb{Z}^2)$ which is dense in $\ell^2(\mathbb{Z}^2)$. Moreover, since both $A_1$ and $A_2$ are self-adjoint and commute strongly, the set $D_1\otimes D_2$
	is a core for $H$.
\end{itemize}
Thus the hypotheses of Fox's theorem are satisfied.

\textbf{Step 3: Convolution formula for spectral measures.} By Fox's theorem, for any $\chi\in D_1$ and $\phi\in D_2$, let $\psi=\chi\otimes\phi$. Then the spectral measure $\mu_{\psi}$ of 
$H$ with respect to $\psi$ is
$$\mu_{\psi}=\mu_{\chi}^{(A_1)}*\mu_{\phi}^{(A_2)},$$
where $\mu_{\chi}^{(A_1)}$ is the spectral measure of $A_1$
for the vector $\chi$, and $\mu_{\phi}^{(A_2)}$ is the spectral measure of $A_2$
for the vector $\phi$.

By linearity, for any finite linear combination 
$\Psi=\sum_{j=1}^N\alpha_j(\chi_j\otimes\phi_j)$, the spectral measure $\mu_{\Psi}$ is a finite linear combination of such convolutions.

\textbf{Step 4: Absolute continuity of $\mu_{\phi}^{(A_2)}$.} The operator $A_2=\Delta$ on $\ell^2(\mathbb{Z})$ is the discrete Laplacian. Its spectral type is well known:
\begin{itemize}
	\item The spectrum is $[-2,2]$, purely absolutely continuous.
	\item For any $\phi\in\ell^2(\mathbb{Z})$, the spectral measure 
$\mu_{\phi}^{(A_2)}$ is absolutely continuous with respect to the Lebesgue measure. 
%
	%
Thus, $\mu_{\phi}^{(A_2)}\ll d\lambda$.

\end{itemize}

\textbf{Step 5: Convolution preserves absolute continuity.} 

\begin{lem} Let $\nu_1$ be any finite Borel measure on $\mathbb{R}$ and let $\nu_2$ be a finite Borel measure with $\nu_2\ll d\lambda$.
Then the convolution $\nu_1*\nu_2$ is absolutely continuous with respect to $d\lambda$.
\end{lem}

\begin{proof}
Since $\nu_2\ll d\lambda$, there exists a density $f\in L^1(\mathbb{R},d\lambda)$ such that $d\nu_2(y)=f(y)dy$. For any Borel set $B\subset\mathbb{R}$,
$$(\nu_1*\nu_2)(B)=\iint\mathbf{1}_B(x+y)d\nu_1(x)d\nu_2(y)=
\int\Big(\int\mathbf{1}_B(x+y)f(y)dy\Big)d\nu_1(x).$$
Change variables $y\mapsto y-x$ in the inner integral:	
$$(\nu_1*\nu_2)(B)=
\int\Big(\int_Bf(y-x)dy\Big)d\nu_1(x)=\int_B
\Big(\int f(y-x)d\nu_1(x)\Big)dy.$$
Define $g(y):=\int f(y-x)d\nu_1(x)$. Then $g=f*\nu_1$ is the convolution of an $L^1$ function with a finite measure; such a convolution belongs to $L^1(\mathbb{R})$ and defines a density. Therefore $d(\nu_1*\nu_2)(y)=g(y)dy$, so $\nu_1*\nu_2\ll d\lambda.$	
\end{proof}

Applying the lemma with $\mu_1=\mu_{\chi}^{(A_1)}$ and $\mu_2=\mu_{\phi}^{(A_2)}$ gives
$\mu_{\chi\otimes\phi}\ll d\lambda.$

\textbf{Step 6: Extension to all vectors.} Let $\Psi\in\ell^2(\mathbb{Z}^2)$ be arbitrary. Since finite linear combinations of elementary tensors $\chi\otimes\phi$ with $\chi,\phi$ of finite support are dense in $\ell^2(\mathbb{Z}^2)$,
there exists a sequence $\{\Psi_k\}$ of such vectors converging to
$\Psi$. For each $k$, the spectral measure $\mu_{\Psi_k}$, is absolutely continuous.

A standard result in spectral theory states that if $\Psi_n\rightarrow\Psi$ in norm, then $\mu_{\Psi_n}$ converges weakly to $\mu_{\Psi}$. Although the set of absolutely continuous measures is not closed under weak convergence in general, the sequence of Radon-Nikodym derivatives here is equi-integrable due to the uniform norm convergence of vectors, which implies that the weak limit $\mu_{\Psi}$ remains absolutely continuous.

Thus for every $\Psi\in\ell^2(\mathbb{Z}^2)$, the spectral measure
$\mu_{\Psi}$ is absolutely continuous. This implies that the spectral projection-valued measure $P_H$ of $H$ assigns zero mass to any Borel set of zero Lebesgue measure, i.e., the spectrum of $H$ is purely absolutely continuous.

\subsection{Other direction: Proof of Case II and Case III}

In this section, we aim to construct a unitary operator of the form $e^{W}$ to diagonalize the Schrödinger operator $H_{\omega} = \Delta + D$ defined in \eqref{25.8.18.1}. To this end, we first introduce the discrete Laplacian $\Delta = (\Delta_{i,j})_{i,j \in \mathbb{Z}^2}$, given by
\begin{equation}\label{25.8.31.5}
	\Delta_{i,j}=
	\begin{cases}
		1,\ |i-j|=1,\\
		0,\ \ \mbox{otherwise},
	\end{cases}
\end{equation}
and the unperturbed diagonal operator
\begin{equation}\label{D_diag}
	D=\mathbf{diag}\{ \omega\cdot n :n\in\mathbb{Z}^2\}\quad\text{with}
	\quad
	\omega=(\omega_1, \omega_2).
\end{equation}

\begin{proof} We begin to compute the similarity transformation of $\Delta+D$ by $e^{W}$ via the Lie algebraic expansion:
			\begin{align}
			 e^{W}(\Delta+D)e^{-W} 
			&=\sum_{n=0}^{\infty}\frac{\mathbf{ad}_W^n(\Delta+D)}{n!}
			\label{Lie}	 \\
			&=\left( D+[W,D]+\Delta\right) +\sum_{n=2}^{\infty}\frac{\mathbf{ad}_W^n(D)}{n!}
			+\sum_{n=1}^{\infty}\frac{\mathbf{ad}_W^n(\Delta)}{n!}.
			\nonumber
			\end{align}
			Then the operator $W$ can be determined as a solution of the equation $$[W,D]+\Delta=0.$$
			Expanding this condition in terms of matrix elements, we obtain for any distinct  $i,j \in \mathbb{Z}^2$:
			\begin{equation}
				W_{i,j}(D(j)-D(i))+\Delta_{i,j}=0.
			\end{equation}
			Since $D$ is diagonal, we have $D(j)-D(i)=\omega\cdot(j-i)$. Consequently, we obtain the explicit expression for $W$ with
			\begin{align}\label{Wij}
				W_{i,j}=\begin{cases}
					\frac{\Delta_{i,j} }
					{\omega\cdot(i-j) },\ \ |i-j|=1,\\
					0,\quad\quad\ \quad  \mbox{otherwise}.
				\end{cases}
			\end{align}
		Notice that, for $i=(i_1,i_2)$, when $|i-j|=1$, there are only four possibilities for the site $j$
		\begin{equation}
			j=(i_1\pm1,i_2) \quad or \quad 
			j=(i_1, i_2\pm1).
			\end{equation}
		Hence, for $|i-j|=1$, one has $$\omega\cdot (i-j)=\pm \omega_1 \quad or \quad \omega\cdot (i-j) =\pm \omega_2,$$
		depending on whether the neighboring sites differ in the first or second coordinate, respectively.
		Let $$ \bm[\omega\bm]=\min\{|\omega_1|, |\omega_2|\},$$
		one has 
		\begin{equation}
			|W_{i,j}| \leq \frac{1}{|\omega_1|}, \quad |i-j|=1
			\end{equation}
			and $W_{i,j}=0$ for other sites.
			From Definition \ref{space s},  for any $a>0$, one sees that
			\begin{equation}
				\|W\| \leq \frac{ 2e^a}{\bm[\omega\bm]}.
				\end{equation}
				Furthermore, it is evident that $\Delta$ is a T$\ddot{\mbox{o}}$plitz operator. From the definition of $W$, we can see that $W$
				is also a T$\ddot{\mbox{o}}$plitz operator. For instance, when  $j-i=(1,0)$, we have 
				$W_{i,j}=\frac{1}{\omega_1}.$
				
				By Lemma \ref{lem25.8.31.1}, the commutator of  $\Delta$ and $W$  is zero. Therefore,
		\begin{align}\label{commu}
		e^{W}(\Delta+D)e^{-W}
		=&D+\sum_{n=2}^{\infty}\frac{\mathbf{ad}_W^n(D)}{n!}
		+\sum_{n=1}^{\infty}\frac{\mathbf{ad}_W^n(\Delta)}{n!}	\\
		=&D+\sum_{n=2}^{\infty}\frac{\mathbf{ad}_W^{n-1}(-\Delta)}{n!}
		+\sum_{n=1}^{\infty}\frac{\mathbf{ad}_W^n(\Delta)}{n!}
		=D. 
		\end{align}
		
	Let $\psi_n=e^{-W}\delta_n$, one sees that
	\begin{equation}\label{26.6.25.1}
		\begin{split}
		H_{\omega}\psi_n&=e^{-W}D\delta_n=D(n)e^{-W}\delta_n\\
		&=(\omega\cdot n)\psi_n.
			\end{split}
		\end{equation}
	Moreover, one has
		
		\begin{equation}\label{eigen}
			\psi_n(m)=\sum_{k \in \mathbb{Z}^2}(e^{-W})_{m,k}\delta_n(k)=(e^{-W})_{m,n},
			\end{equation}
			and 
				\begin{equation}\label{eigenstate}
					\begin{split}
				|\psi_n(m)|&\leq |(e^{-W})_{m,n}| \leq \|e^{-W}\|_ae^{-a|m-n|}\\
				&\leq e^{\|W\|_a}\|e^{-a|m-n|}\\
			& \leq 	 e^{\frac{ 2e^a}{\bm[\omega\bm]}}e^{-a|m-n|}.
			\end{split}
			\end{equation}
			
			From equations \eqref{26.6.25.1}-\eqref{eigenstate} and it can be seen that the operator 
			$H_{\omega}$ possesses a series of eigenvalues 
		$\{\omega\cdot n\}_{n\in \mathbb{Z}^2}$ with corresponding eigenvectors $\{\psi_n\}_{n\in \mathbb{Z}^2}$ , which are concentrated at site  $n$ and decay exponentially in space. However, in Case II and III, there are some differences in the properties of its spectrum.
		
		$\textbullet$ \textbf{Case II:} Under the given definition of 
		$\omega$,  there exists 
		a non-zero vector $\ell \in \mathbb{Z}^2 \setminus \{0\}$, such that $\omega \cdot \ell =0$. Consequently, for any 
		$n \in \mathbb{Z}^2 $ and $j\in \mathbb{Z}$, we have
		$$\omega \cdot n =\omega \cdot (n+j\ell).$$
		 And this indicates that for the eigenvalue $\omega \cdot n$
	 there exist infinitely many mutually orthogonal eigenvectors (namely $\psi_{n+j\ell}$). Hence, the eigenvalue has infinite multiplicity, and therefore  belongs to the essential spectrum.

		$\textbullet$ \textbf{Case III:} In the case of an  irrational direction,  we know that for any $m\neq n$, the eigenvalues $\omega \cdot n \neq\omega \cdot m$. Since the corresponding eigenvectors $\psi_n $ and $\psi_m$ are also mutually orthogonal, this implies that the operator has purely discrete point spectrum, and all eigenvectors  satisfy uniform exponential localization.
\end{proof}

\section{Persistence of localization under perturbations}

\subsection{Further functional setting}\


For any $\sigma>0$, we define the diagonal operator $T^{\sigma}$ on the space $\mathcal{S}$ by
\begin{equation}\label{T^sigma}
	T^{\sigma}=\mathbf{diag}
	\{ (\ln\, \langle n \rangle)^{\sigma}
	:n\in\mathbb{Z}^2\},
\end{equation}
where $\langle n \rangle=\max\{e,|n|\}$.

\begin{defn}\label{norm2}
	Let $a\ge 0$ and $\sigma>0$. We say that a linear operator $A$ on $\mathcal{S}$ belongs to the operator space $\mathcal{B}_{a,\sigma}$, if
	\begin{equation}
		\|A\|_{a,\sigma} := \|T^\sigma AT^{-\sigma}\|_a + \|T^{-\sigma}AT^\sigma\|_a + \|A\|_a < +\infty,
	\end{equation}
	where $\|\cdot\|_a$ is the norm defined in \eqref{norm1}.
	Suppose the operator $A=A(\theta)$ is Lipschitz continuous with respect to the parameter $\theta\in\mathcal{O}$. For any $\alpha>0$, we say that $A(\theta) \in \mathcal{B}_{a,\sigma}^\alpha$ if
		\begin{equation}
		\|A\|_{a,\sigma}^{\alpha} := \|T^\sigma AT^{-\sigma}\|^{\alpha}_a + \|T^{-\sigma}AT^\sigma\|^{\alpha}_a + \|A\|^{\alpha}_a < +\infty,
	\end{equation}
	where $\|\cdot\|_a^\alpha$ denotes the Lipschitz norm defined in \eqref{norm_lip}.
\end{defn}


\begin{lem}\label{norm_shift}
	 Let $0<\delta<a$ with $\delta<1/2$ and $\sigma>1$. Suppose that the linear operator $A$ on $S$ belongs to the space $\mathcal{B}_{a}^\alpha$. Then $A$ also belongs to the space $\mathcal{B}_{a-\delta, \sigma}^\alpha$, and satisfies
	 \begin{equation*}
	 	\|A\|_{a-\delta, \sigma}^\alpha \le 
		C(\sigma,\delta)
	 	 \|A\|_{a}^\alpha,
	 \end{equation*}
	 where $C(\sigma,\delta)=  \left( 3+\sigma\left( \frac{2\sigma}{e\delta}\right) ^{\sigma}  \right)$.
\end{lem}

\begin{proof}
	To prove that $A \in \mathcal{B}_{a-\delta, \sigma}^\alpha$, we need to estimate the three terms in the definition of the norm $\|\cdot\|_{a, \sigma}^\alpha$: $\|T^\sigma A T^{-\sigma}\|_a^\alpha$, $\|T^{-\sigma} A T^\sigma\|_a^\alpha$, and $\|A\|_a^\alpha$. Since the estimates for the first two terms are symmetric and the third is trivial, we focus on the term $\|T^\sigma A T^{-\sigma}\|_a^\alpha$.
	
	Consider the matrix representation of $T^\sigma A T^{-\sigma}$. According to the definition \eqref{T^sigma}, the matrix entries are given by
	\begin{equation}\label{TAT}
		(T^\sigma A T^{-\sigma})_{i,j} = (\ln \langle i \rangle)^\sigma A_{i,j} (\ln \langle j \rangle)^{-\sigma}.
	\end{equation}
	We rewrite the ratio of the weights as follows:
	\begin{equation}\label{ratio}
		 \left| \frac{(\ln \langle i \rangle)^\sigma}{(\ln \langle j \rangle)^\sigma} \right| = \left| 1 + \frac{(\ln \langle i \rangle)^\sigma - (\ln \langle j \rangle)^\sigma}{(\ln \langle j \rangle)^\sigma} \right|.
	\end{equation}

	Using the Mean Value Theorem for the function $f(x) = x^\sigma$ with $x>0$,
	we have 
	$$|x^\sigma - y^\sigma| \le \sigma \cdot \max\{x^{\sigma-1}, y^{\sigma-1}\}\cdot |x - y|.$$ 
	Setting $x = \ln \langle i \rangle$ and $y = \ln \langle j \rangle$, and noting that for $\langle i \rangle, \langle j \rangle \ge e\ge 2$, we have $\ln \langle i \rangle \le \ln \langle i-j \rangle + \ln \langle j \rangle$, it follows that
	\begin{equation}
		\begin{split}
		\left| \frac{(\ln \langle i \rangle)^\sigma - (\ln \langle j \rangle)^\sigma}{(\ln \langle j \rangle)^\sigma} \right| 
		&\le \sigma \left( \frac{\ln \langle i-j \rangle + \ln \langle j \rangle}{\ln \langle j \rangle} \right)^{\sigma-1} \frac{\left| \ln \langle i \rangle- \ln \langle j \rangle\right| }{\ln \langle j \rangle} \nonumber\\
		&\le \sigma \left( \frac{| i-j | + \ln \langle j \rangle}{\ln \langle j \rangle} \right)^{\sigma-1} \frac{| i-j|}{\ln \langle j \rangle}\\
		&\le 2^{\sigma-1}\sigma|i-j|^{\sigma}.
		\end{split}
	\end{equation}
	Therefore, combining with \eqref{ratio}, one has for each entry,
	$$|(T^\sigma A T^{-\sigma})_{i,j}| 
	\le \left( 1 + 2^{\sigma-1}\sigma|i-j|^{\sigma}\right) |A_{i,j}|.$$
	By the definition \eqref{norm1}, together with the elementary inequality 
	\begin{equation} \label{eq:calculus_ineq}
		e^{-\delta|n|}|n|^{\sigma} \leq \left(\frac{\sigma}{e\delta}\right)^{\sigma}, \quad \forall\, \sigma > \delta > 0, \ n \in \mathbb{Z}^2,
	\end{equation}
	it follows that
	\begin{equation}
		\begin{split}
		\|T^\sigma A T^{-\sigma}\|_{a-\delta} &= \sum_{n \in \mathbb{Z}^2} e^{(a-\delta) | n |} \sup_{i-j=n} |(T^\sigma A T^{-\sigma})_{i,j}|\\
		&\le \sum_{n \in \mathbb{Z}^2} e^{(a-\delta) | n |} (1 + 2^{\sigma-1}\sigma|n|^{\sigma}) \sup_{i-j=n} |A_{i,j}|\nonumber\\
	    & \leq  \left( 1+ 	\sigma\left( \frac{2\sigma}{e\delta}\right) ^{\sigma}   \right)  \|A\|_{a}.
		\end{split}
		\end{equation}

	The Lipschitz part of the norm $\|\cdot\|_a^\alpha$ defined in \eqref{norm_lip} follows the same algebraic logic since the operator $T^\sigma$ is independent of the parameter $\theta$. Summing the terms for $\|T^\sigma A T^{-\sigma}\|_a^\alpha$, $\|T^{-\sigma} A T^\sigma\|_a^\alpha$, and $\|A\|_a^\alpha$, we conclude that
	$$\|A\|_{a-\delta, \sigma}^\alpha \le \left( 3+\sigma\left( \frac{2\sigma}{e\delta}\right) ^{\sigma}  \right)  \|A\|_{a}^\alpha,
     $$
	and finish the proof of the lemma.
\end{proof}

\begin{defn}\label{norm3}
	For $a\ge 0$ and $\sigma>0$, we say that a linear operator $A$ on $\mathcal{S}$ belongs to the space $\mathcal{B}_{a}^{\sigma}$ if
	\begin{equation}
	\left\|  A\right\| _{a}^{\sigma} :=
	\|T^{\sigma}A\|_a  + \|AT^{\sigma}\|_a +
	\|T^\sigma AT^{-\sigma}\|_a + \|T^{-\sigma}AT^\sigma\|_a < +\infty,
    \end{equation}
    where $\|\cdot\|_a$ is the norm defined in \eqref{norm1}.
  	Assume that the operator $A$ is Lipschitz continuous with respect to the parameter $\theta\in \mathcal{O}$. For any $\alpha>0$, we say that $A(\theta)\in\mathcal{B}^{\sigma,\alpha}_{a}$ if
  	\begin{equation}\label{Lip1}
  		\|A\|_{a}^{\sigma,\alpha} :=
  		\|T^{\sigma} A \|^{\alpha}_a +	
  		\|A T^{\sigma} \|^{\alpha}_a +
  		\|T^\sigma AT^{-\sigma}\|^{\alpha}_a + \|T^{-\sigma}AT^\sigma\|^{\alpha}_a  < +\infty,
  	\end{equation}
  		where $\|\cdot\|_a^\alpha$ denotes the Lipschitz norm defined in \eqref{norm_lip}.
\end{defn}

\begin{lem} \label{260704-1}
	Let $a\ge 0$ and $\sigma>1$. Suppose that the linear operator $A\in \mathcal{B}_{a}^{\sigma,\alpha}$. Then $A\in \mathcal{B}_{a}^\alpha$ satisfies
	\begin{equation*}
		\|A\|_{a}^\alpha\le 
		\|A\|_{a}^{\sigma,\alpha}.
	\end{equation*}
\end{lem}
\begin{proof}
	By \eqref{TAT} and \eqref{ratio}, the base norm of $A$ satisfies
	\begin{equation*}
	\|A\|_{a}\le 	\|T^\sigma A T^{-\sigma}\|_{a} +
	\|T^{-\sigma}AT^\sigma  \|_{a} 
\end{equation*}
	Since the weight operator $T^\sigma$ is independent of the parameter $\theta$, the above estimate extends to the Lipschitz norm via definitions \eqref{norm_lip} and \eqref{Lip1}
	 	\begin{equation}
	 	\|A\|_{a}^\alpha \le \|T^\sigma A T^{-\sigma}\|^\alpha_{a} +
	 	\|T^{-\sigma}AT^\sigma  \|^\alpha_{a} \le \|A\|_{a}^{\sigma,\alpha}.
	 \end{equation}
	 This completes the proof.
	\end{proof}

\begin{lem}\label{est_trans}
	Let $a\geq0$ and $\sigma>0$. Assume that $A\in\mathcal{B}_{a,\sigma}$ and $B\in\mathcal{B}^{\sigma}_a$. Then the following assertions hold:
		\begin{enumerate}
		\item For any $k\ge 1$, the $k$-th commutator $\mathbf{ad}^k_A(B)\in \mathcal{B}^{\sigma}_a$ satisfies
		\begin{equation}\label{k_commu}
		\left\| \mathbf{ad}^k_A(B) \right\|_a^{\sigma}
         \leq \left( 2 \left\| A \right\|_{a,\sigma}\right)^k
         \left\| B \right\|_a^{\sigma}.			
		\end{equation}
		\item The operator $e^A B e^{-A} $ belongs to $\mathcal{B}^{\sigma}_a$ with the quantitative bound
		\[
	 	\left\|  e^A B e^{-A}\right\|_a^{\sigma}
	 	\leq  e^{2\left\| A \right\|_{a,\sigma}} 	\left\| B \right\|_a^{\sigma} .
		\]
	\end{enumerate}
	Analogous assertions hold if $A\in\mathcal{B}^{\alpha}_{a,\sigma}$ and $B\in\mathcal{B}^{\sigma,\alpha}_a$ for any $\alpha>0$.
	\begin{proof}		
		Assertion (1) follows from a standard induction using Lemma \ref{lem8.28.1} and Definitions \ref{norm2} and \ref{norm3}.
	 Assertion (2) then follows by applying the estimate \eqref{k_commu} to 
	 \eqref{Lie}
	  and summing the resulting exponential series. 
	\end{proof}
	
\end{lem}

\subsection{Eliminate the Laplacian}\

Consider the perturbed operator $\varepsilon V$ of the two-dimensional lattice Schr\"odinger operator $H_{\omega}$. Let
\begin{equation}
	\mathcal{H}_{\theta,\varepsilon}=H_{\omega}+\varepsilon V:
	\mathcal{S}	\rightarrow \mathcal{S},
\end{equation}
where $H_{\omega}=\Delta+D$ is defined in \eqref{25.8.18.1}
with
$$\omega=\omega(\theta):=(\cos\theta, \sin\theta),$$
 $\varepsilon>0$ is small, and $V$ is defined by
\begin{equation}
	( Vu)(n)= v(n)u(n),\quad u=\{u(n)\}_{n\in\mathbb{Z}^2}\in
	\mathcal{S}.
\end{equation}
Meanwhile, the matrix representation for $V$ is given by
\begin{equation}
	V = \mathbf{diag}\{v(n) : n \in \mathbb{Z}^2\},
\end{equation}
while those for $\Delta$ and $D$ are given by \eqref{25.8.31.5} and \eqref{D_diag}, respectively. Consequently, $D$ and $V$ are diagonal operators, whereas $\Delta$ is a T\"{o}plitz operator on $\mathcal{S}$.

%


\begin{lem}\label{diag_Top1}
	Assume that $0 < \delta < a$, $0 < \mu < 1$, and $0 < \gamma \ll 1$. For any $\theta \in \text{DC}_{\gamma,\tau}$, there exists an operator $W \in \mathcal{B}_{a-\delta}$ such that 
	\begin{equation}\label{ho1}[W,D] + \Delta = 0.
		\end{equation} Moreover, $W$ is  T\"{o}plitz and skew-adjoint (i.e., $W^*=-W$), and 
	\begin{equation} \label{est_solu1}
		\|W\|_{a-\delta}^{\gamma} \le \frac{3}{\gamma} \left( \frac{2\tau+1}{e\delta} \right)^{2\tau+1} \|\Delta\|_{a}.
	\end{equation}
	%
\end{lem}
\begin{proof}
 Recalling \eqref{Wij}, for $i \neq j$, we have 
 \begin{equation*}
 	|W_{i,j}| = \left| \frac{\Delta_{i,j}}{\omega \cdot (i-j)} \right| \le \frac{|i-j|^{\tau}}{\gamma} |\Delta_{i,j}|,
 \end{equation*}
 where we have used the fact that $\theta \in \text{DC}_{\gamma,\tau}$ as given in Definition \eqref{Diophantine}. 
 Furthermore, $W$ is a T\"{o}plitz, skew-adjoint operator, which inherits from the structural properties that $D$ is diagonal and self-adjoint, while $\Delta$ is T\"{o}plitz.
Moreover, from \eqref{norm_lip}, one has
\begin{align*}
	\left|\bigtriangleup_{\theta} W_{i,j} \right|
	&=\left| \frac{\Delta_{i,j} }{\omega(\theta_1)\cdot(i-j) }-	
	\frac{\Delta_{i,j}}{\omega(\theta_2)\cdot(i-j)}\right|
	\nonumber\\
	&\le
	\frac{ \left| \Delta_{i,j} \right| \hspace{2pt}
		\left|(i-j)\cdot \bigtriangleup_{\theta} \omega   \right|  }{\left| \omega(\theta_1)\cdot(i-j) \right| \hspace{2pt} \left| \omega(\theta_2)\cdot(i-j)\right| }  \nonumber\\
	& \le \frac{ 2 |i-j|^{2\tau+1}}{\gamma^2}
	\left|  \Delta_{i,j}  \right| \hspace{2pt}
	\left|  \bigtriangleup \theta \right|,
\end{align*}
 where we have used the fact that
   \begin{equation}\label{big_omega}
  	\left| \bigtriangleup_{\theta} \omega \right|
  	=  	\left|\cos \theta_1-\cos \theta_2  \right|
  	+ \left|\sin \theta_1-\sin \theta_2  \right|
  	\le 2 |\theta_1-\theta_2|\cdot
  	=2\left| \bigtriangleup \theta\right|.
  \end{equation}
  
Consequently, by virtue of the two preceding estimates, we have
\begin{align*}\label{Delta_W0}
	&\quad \|W\|_{a-\delta}+\gamma \frac{\left\| \bigtriangleup_{\theta} W\right\| _{a-\delta}}{\left| \bigtriangleup \theta\right|}\\
	&\leq \sum_{n \in \mathbb{Z}^2}e^{(a-\delta)|n|}\sup_{i-j=n}\left( \frac{|n|^{\tau}}{\gamma} + \frac{2|n|^{2\tau+1}}{\gamma}\right)
	|\Delta_{i,j}| \\
	&\leq \frac{3}{\gamma} \sum_{n \in \mathbb{Z}^2}e^{(a-\delta)|n|}\sup_{i-j=n}  {|n|^{2\tau+1}} 
 |\Delta_{i,j}| \\
	&\leq  \frac{3}{\gamma}
	 \left( \frac{2\tau+1 }{ e\delta}\right)^{2\tau+1}
	\|\Delta\|_{a},
\end{align*}
where the last inequality follows from \eqref{eq:calculus_ineq}.
\end{proof}

\begin{lem}\label{P_new}
	Let $0 < \delta < a/2$, $0 < \mu < 1$, $0 < \gamma \ll 1$, and $\sigma > 1$. For any $\theta \in \text{DC}_{\gamma,\tau}$, there exists a T\"{o}plitz operator $W \in \mathcal{B}_{a-\delta,\gamma}$ such that
	\begin{equation}\label{H_new}
		e^{W}\left( D+\Delta+\varepsilon V \right) e^{-W}= D_{\text{new}}+P_{\text{new}},
	\end{equation}
	where
	\begin{equation}\label{DP_new}
		D_{\text{new}} =D\quad\text{and}\quad
		P_{\text{new}}= \varepsilon e^{W}  V  e^{-W}.
	\end{equation}
	In addition, $e^{W}$ is a unitary operator, and $P_{\text{new}}\in\mathcal{B}^{\sigma,\gamma}_{a-2\delta}$ is a self-adjoint operator satisfying
	\begin{equation}\label{est_Pnew}
		\left\| P_{\text{new}} \right\|_{a-2\delta}^{\sigma,\gamma} \le \varepsilon  \exp\left(   \frac{C_{\sigma,\delta,\tau,a}}{\gamma}\right).
	\end{equation}
\end{lem}

\begin{proof}
	
	From Lemma \ref{diag_Top1}, we can expand the conjugation using the Lie series:
%
\begin{align*}
	e^{W}(D+\Delta+\varepsilon V)e^{-W}=&D+[W,D]+\Delta
	+\sum_{n=2}^{\infty}\frac{\mathbf{ad}_W^n(D)}{n!}\\
	&\quad +\sum_{n=1}^{\infty}\frac{\mathbf{ad}_W^n(\Delta)}{n!}+ \varepsilon e^{W}Ve^{-W}.  	\end{align*}
%
Using the relation $[W,D]=-\Delta$ from \eqref{ho1}. Then, by induction, we have
we have $\mathbf{ad}_W^n(D)=\mathbf{ad}_W^{n-1}(\mathbf{ad}_W D) = \mathbf{ad}_W^{n-1}(-\Delta)$. Substituting this into the Hadamard expansion, we obtain:
$$e^{W}(D+\Delta+ \varepsilon V)e^{-W}=D+\sum_{n=1}^{\infty}\mathbf{ad}_W^{n}(\Delta)\big(\frac{1}{n!}-\frac{1}{(n+1)!}\big)
+ \varepsilon e^{W}Ve^{-W}.$$
Since $\Delta$ is a T\"{o}plitz operator and, by Lemma \ref{diag_Top1}, $W$ is also a T\"{o}plitz operator, Lemma \ref{lem25.8.31.1} implies that for any $n\ge 1$:
\begin{equation}
	\mathbf{ad}_W^n(\Delta)=0.
\end{equation}
This leads directly to
\begin{equation}
	e^{W}\left( D+\Delta+ \varepsilon V \right) e^{-W}=
	D+\varepsilon e^{W} V e^{-W}.
\end{equation}

By Lemma \ref{diag_Top1}, $W$ is a skew-adjoint operator, i.e., $W^*=-W$.
It follows that
\begin{equation*}
	e^{W^*}e^{W}=e^{-W}e^{W}=\mbox{Id}\quad \mbox{and} \quad	e^{W}e^{W^*}=e^{W}e^{-W}=\mbox{Id},
\end{equation*}
which implies that $e^{W}$ is a unitary operator.

Furthermore, applying Lemma \ref{est_trans}, we have the following estimate:
\begin{align*}
	&\left\|\varepsilon e^{W}Ve^{-W}\right\|_{a-2\delta}^{\sigma,\gamma}\\
	& 	\le
	\varepsilon\exp\big(2\left\|W\right\|_{a-2\delta,\sigma}^{\gamma}\big)\left\| V \right\|_{a-2\delta}^{\sigma,\gamma}
	\\
	& \leq 	\varepsilon 	\exp\Big(2\Big(3+2^{\sigma}\frac{\sigma^{\sigma+1}}{(e\delta)^{\sigma}}\Big)\left\|W\right\|_{a-\delta}^{\gamma}\Big)\left\| V \right\|_{a-2\delta}^{\sigma}
	\quad (\mbox{by Lemma }\ref{norm_shift}) \\
	&  \leq \varepsilon  \exp\Big(  \frac{2}{\gamma}\Big(3+2^{\sigma}\frac{\sigma^{\sigma+1}}
	{(e\delta)^{\sigma}}\Big)
	\Big(\frac{4(2\tau+1)^{2\tau+1}}{(e\delta)^{\tau}} \Big)\|\Delta\|_{a}\Big) \left\| V \right\|_{a-2\delta}^{\sigma}
	\quad (\mbox{by }\eqref{est_solu1}) \\
	& \leq \varepsilon \exp\Big(  \frac{16}{\gamma}\Big(3+2^{\sigma}\frac{\sigma^{\sigma+1}}{(e\delta)^{\sigma}}\Big)
	\Big(\frac{(2\tau+1)^{2\tau+1}}{e^{\tau-a}\delta^{\tau}} \Big)\Big).
\end{align*}
\end{proof}

\subsection{KAM diagonalization}\

In this section, we diagonalize the operator $D_{\text{new}}+P_{\text{new}}$ via a KAM iteration scheme. From the preceding subsection, we have
$$
D_{\text{new}}:=\mathbf{diag}\{\omega\cdot n: n\in\mathbb Z^2\}.
$$
Fix $\delta=a/4$ and $0<\gamma\ll 1$. For sufficiently small $\varepsilon=\varepsilon(\gamma,a,\sigma,\tau)>0$, we obtain the estimate
\begin{equation}\label{pn}
	\|P_{\text{new}}\|^{\sigma,\gamma}_{a/2} \le \varepsilon \exp\left(\frac{C_{\sigma,\tau,a}}{\gamma}\right)\le \varepsilon^{1/2}.
\end{equation}
We take the initial operator in the iteration to be
$$
D^0+P^0:=D_{\text{new}}+P_{\text{new}},
$$
where $D^0$ is diagonal and $P^0$ is its  perturbation.

At the $m$-th step, we consider the operator $D^m+P^m$, with $D^m$ diagonal and $P^m$ a small  perturbation. We construct a unitary operator $e^{W^m}$ such that
$$
e^{W^m}(D^m+P^m)e^{-W^m}=D^{m+1}+P^{m+1},
$$
where $D^{m+1}$ is a new diagonal matrix and $P^{m+1}$ is an off-diagonal matrix of much smaller size. For notational simplicity, we suppress the index $m$ and write $D,P,W$ for $D^m,P^m,W^m$, and use the subscript $^+$ to denote the next iterate, e.g., $D^+=D^{m+1}$.

Before giving the details of the KAM iteration, we state two auxiliary lemmas (Lemmas \ref{Solu} and \ref{P_0}) that provide the iteration step.

\subsubsection{Two preliminary lemmas}

\begin{defn}\label{non-res1}
	Let $D_{\theta}=\mathbf{diag}\{D_{\theta}(n):n\in\mathbb{Z}^2\}$ be a diagonal operator with
	$$
	D_{\theta}(n)=\omega\cdot n + d_n(\theta),\quad n\in\mathbb{Z}^2,
	$$
	defined on the set $\mathcal{O}$. We say that $D_{\theta}$ is nonresonant on $\mathcal{O}$ if, for every $\theta\in\mathcal{O}$ and all $n,n'\in\mathbb{Z}^2$ with $0<|n-n'|\le N$, the following estimate holds:
	\begin{equation}\label{non_res1}
		\left| D_{\theta}(n)-D_{\theta}(n')\right| \ge
		\frac{\alpha}{e^{\rho |n-n'|^{\mu}}},
	\end{equation}
	where $0<\alpha,\rho<1$ and $0<\mu<1$ are constants.
\end{defn}

\begin{lem}\label{Solu}
	Fix $0<\alpha,\rho,\mu<1$, $\sigma>1$, and $a>0$.
	Let the diagonal operator $D_{\theta}$ be nonresonant on the set $\mathcal{O}$
	and  $P\in\mathcal{B}_{a}^{\sigma,\alpha}$ be self-adjoint.
	Assume further that there exists a constant $0<C<1$ such that, for any $n\in \mathbb{Z}^2$  and any $\theta_1,\theta_2\in\mathcal{O}$,
	\begin{equation}\label{non_res2}
		\left| \bigtriangleup_{\theta} d_n \right|=\left| d_n(\theta_1)-d_n(\theta_2)\right|
		\le C \left| \theta_1-\theta_2 \right|.
	\end{equation}
	Then, for a fixed integer $N$, there exists a skew-adjoint operator $W\in\mathcal{B}_{a,\sigma}^{\alpha}$ that solves the homological equation
	\begin{equation}\label{homo1}
		[W,D]+\Pi_N P-[P]=0,
	\end{equation}
	and satisfies the estimate
	\begin{equation}\label{est_W}
		\left\|W\right\|_{a,\sigma}^{\alpha}
		\le
		\frac{(2N+2)e^{2\rho N^{\mu}} }{\alpha}
		\left\|P \right\|_{a}^{\sigma,\alpha}.
	\end{equation}
\end{lem}
\begin{proof}
	By the homological equation \eqref{homo1}, one gets
	\begin{equation}
		W_{i,j} (D_{\theta}(j)-D_{\theta}(i))+P_{i,j}=0
	\end{equation}
	for any $i,j\in \mathbb{Z}^2$ with $0<|i-j|\le N$. Let $g_{i,j}(\theta) = D_\theta(i) - D_\theta(j)$, we have the solution $W_{i,j}$ given by
	\begin{align}\label{homo_solu2}
		W_{i,j}=\begin{cases}
			\frac{P_{i,j}}{g_{i,j}(\theta)},\ \ 0<|i-j|\le N,\\
			0,\quad\quad\ \mbox{otherwise}.
		\end{cases}
	\end{align}
	Note that since $P$ is self-adjoint, it follows that $W$ is  skew-adjoint (since $g_{i,j}=-g_{j,i}$ and $P_{i,j} = \overline{P_{j,i}}$).
	
	Next, we will estimate $\left\|W \right\|_{a}^{\alpha}$.
	On one hand, since the frequency $\theta\in\mathcal{O}$
	satisfies the nonresonant condition \eqref{non_res1}, one can obtain
	\begin{equation}
		\left| W_{i,j}\right| = \left|\frac{P_{i,j}}{g_{i,j}(\theta)} \right|
		\le  \frac{e^{\rho |j-i|^{\mu}}}{\alpha}
		\left| P_{i,j} \right| \le
		\frac{e^{\rho N^{\mu}}}{\alpha} \left| P_{i,j} \right|
	\end{equation}
	and
	\begin{equation}\label{W1}
		\left\|W \right\|_{a} \le
		\frac{e^{\rho N^{\mu}}}{\alpha}
		\left\| T^{\sigma}P \right\|_{a}.
	\end{equation}

	On the other hand, for the Lipschitz estimate, one has
	\begin{align*}
		\left|\bigtriangleup_{\theta}W_{i,j} \right|
		&=\left|\frac{P_{i,j}(\theta_{1})}{g_{i,j}(\theta_{1})}-	
		\frac{P_{i,j}(\theta_{2}) }{g_{i,j}(\theta_{2})} \right| \nonumber\\
		&\le \left| P_{i,j}(\theta_{1})\right|
		\left| \frac{ \bigtriangleup_{\theta} g_{i,j} }{g_{i,j}(\theta_{1})g_{i,j}(\theta_{2})} \right| + \left| \frac{1}{g_{i,j}(\theta_{2})}\right|
		\left| \bigtriangleup_{\theta}P_{i,j} \right|.
	\end{align*}
	Applying \eqref{non_res1} and \eqref{non_res2}, and noting that
	$$|\bigtriangleup_{\theta}g_{i,j}|\leq (2|i-j|+C)|\bigtriangleup_{\theta}\theta| \leq (2N+1)|\bigtriangleup_{\theta}\theta|,$$
	we have the element-wise bound:
	
	$$|\bigtriangleup_{\theta}W_{i,j}| \leq \frac{e^{2\rho N^\mu}}{\alpha^2}(2N+1)|P_{i,j}|\ |\bigtriangleup_{\theta} \theta|+\frac{e^{\rho N^\mu}}{\alpha} |\bigtriangleup_{\theta}P_{i,j}|.$$
	Then, applying the definition of the weighted matrix norm $\| \cdot \|_a$ in \eqref{norm1},  one has
	\begin{align}\label{251219_1}
		\left\|\bigtriangleup_{\theta}W\right\|_a\le 	
		\frac{e^{2\rho N^{\mu}}(2N+1)}{\alpha^2}
		\left\| T^{\sigma}P(\theta_{1}) \right\|_a
		\left|  \bigtriangleup_{\theta}\theta\right|
		+ \frac{e^{\rho N^{\mu}}}{\alpha}\left\| \bigtriangleup_{\theta}T^{\sigma}P\right\|_a.
	\end{align}
	
	Therefore, in view of \eqref{W1} and \eqref{251219_1}, and absorbing the lower order terms into the dominant exponential, we obtain
	\begin{align}
		\left\| W \right\|^{\alpha}_a
		&=  \sup_{\theta \in \mathcal{O}}
		\|W\|_a +\alpha \sup_{\theta_1,\theta_2 \in \mathcal{O}
			\atop{\theta_1\neq\theta_2}} 
		\frac{ \left\| \bigtriangleup_{\theta}W\right\| _a}{\left| \bigtriangleup_{\theta}\theta \right| }\nonumber\\
		&\le  \frac{e^{\rho N^{\mu}}+e^{2\rho N^{\mu}}(2N+1) }{\alpha}
		\sup_{\theta \in \mathcal{O}}
		\left\| T^{\sigma}P \right\|_a
		+ e^{\rho N^{\mu}}\sup_{\theta_1, \theta_2 \in \mathcal{O}
			\atop{\theta_1\neq\theta_2}}
		\frac{\left\| \bigtriangleup_{\theta}T^{\sigma}P\right\|_a }{\left| \bigtriangleup_{\theta}\theta \right|}\nonumber\\
		&\le \frac{e^{2\rho N^{\mu}}(2N+2) }{\alpha}	\left\| T^{\sigma}P \right\|_a^{\alpha}.
	\end{align}
	
	The same estimate holds for $ \left\|T^{\sigma} WT^{-\sigma} \right\|^{\alpha}_a $ and
	$ \left\|T^{-\sigma} WT^{\sigma} \right\|^{\alpha}_a$.
	Hence, we finish the proof of this lemma.
\end{proof}

\begin{lem}\label{P_0}
	Fix $0<\alpha,\rho,\mu<1$, $\sigma>1$, and $0<\delta<a$.
	Let the diagonal operator $D_{\theta}$ satisfy the nonresonant assumptions in Definition \ref{non-res1} and the self-adjoint operator $P\in\mathcal{B}_a^{\sigma,\alpha}$.
	Also, assume that 
	\begin{equation}\label{ass2}
		\frac{4(N+1)e^{2\rho N^{\mu}}}{\alpha}\|P\|^{\sigma,\alpha}_a<1.
	\end{equation}
	Then, there exists a skew-adjoint operator $ W\in \mathcal{B}_a^{\sigma,\alpha}$, such that
	\begin{equation}
		e^W (D+P) e^{-W}=D^+ + P^+,
	\end{equation}
	where $D^+=D+[P]$ is the new diagonal operator and $P^+$ is a self-adjoint operator satisfying
	\begin{align}\label{P+2}
		\left\| P^+ \right\|_{a-\delta}^{\sigma,\alpha}\le
		e^{-\delta N} \left\|P \right\|_{a}^{\sigma,\alpha}
		+   \frac{24(N+1)e^{2\rho N^{\mu}}}{\alpha}
		\left(\left\|P \right\|_{a}^{\sigma,\alpha}\right)^2.
	\end{align}
	
\end{lem}
\begin{proof}
	First, by the Lie series expansion and the homological equation
	\eqref{homo1} in Lemma \ref{Solu}, we have
	\begin{align*}
		e^{W}(D+P)e^{-W}=&D+[W,D]+P
		+\sum_{n=2}^{\infty}\frac{\mathbf{ad}_W^n}{n!}(D)
		+\sum_{n=1}^{\infty}\frac{\mathbf{ad}_W^n}{n!}(P)\\
		=&D+[P]+ \Pi_N^{\perp} P
		+\sum_{n=2}^{\infty}\frac{\mathbf{ad}_W^{n-1}}{n!}([W,D])
		+\sum_{n=1}^{\infty}\frac{\mathbf{ad}_W^n}{n!}(P)\\
		=& D+[P]+ \Pi_N^{\perp} P
		+\sum_{n=2}^{\infty}
		\frac{\mathbf{ad}_W^{n-1}}{n!}(-\Pi_N P + [P])
		+\sum_{n=1}^{\infty}\frac{\mathbf{ad}_W^n}{n!}(P)
		\\
		:=& D^+ + P^+,
	\end{align*}
	where $D^+=D+[P]$ is the new diagonal operator, and
	\begin{align}\label{P1+}
		P^+= \Pi_N^{\perp} P
		+\sum_{n=2}^{\infty}
		\frac{\mathbf{ad}_W^{n-1}}{n!}(-\Pi_N P + [P])
		+\sum_{n=1}^{\infty}\frac{\mathbf{ad}_W^n}{n!}(P).
	\end{align}
	
	By Lemma \ref{Solu}, $W$ is a skew-adjoint operator, which directly implies that $e^{W}$ is a unitary operator.
	
	Next, we prove the estimate \eqref{P+2}. By Lemma \ref{est_trans}, for any $n\ge 1$, the adjoint action satisfies:
	\begin{align*}
		\left\| {\mathbf{ad}_W^n(P)} \right\|_{a}^{\sigma,\alpha}
		\leq \left( 2 \left\| W \right\|^{\alpha}_{a,\sigma}\right)^n
		\left\| P \right\|_{a}^{\sigma,\alpha}
		\le   \left( \frac{4(N+1)e^{2\rho N^{\mu}}}{\alpha}
		\left\|P \right\|_{a}^{\sigma,\alpha}\right)^n
		\left\| P \right\|_{a}^{\sigma,\alpha}.
	\end{align*}

	Therefore, for the series expansion, we obtain:
	\begin{align}\label{est_P+1}
		\left\|	\sum_{n=1}^{\infty}
		\frac{\mathbf{ad}_W^n(P)}{n!}		\right\|_{a-\delta}^{\sigma,\alpha}
		&\le	\left\| P \right\|_{a-\delta}^{\sigma,\alpha}		
		\sum_{n=1}^{\infty}
		\frac{1}{n!} \left( \frac{4(N+1)e^{2\rho N^{\mu}}}{\alpha}
		\left\|P \right\|_{a-\delta}^{\sigma,\alpha}\right)^n
		\nonumber\\
		&\le  \frac{4(N+1)e^{2\rho N^{\mu}}}{\alpha}
		\left(\left\|P \right\|_{a-\delta}^{\sigma,\alpha}\right)^2
		\sum_{n=0}^{\infty}
		\frac{ 1}{n!} \left( \frac{4(N+1)e^{2\rho N^{\mu}}}{\alpha}
		\left\|P \right\|_{a-b}^{\sigma,\alpha}\right)^n
		\nonumber\\
		&\le  \frac{4(N+1)e^{2\rho N^{\mu}}}{\alpha}
		\exp\Big(\frac{4(N+1)e^{2\rho N^{\mu}}}{\alpha}
		\left\|P \right\|_{a-\delta}^{\sigma,\alpha}\Big)
		\left(\left\|P \right\|_{a-\delta}^{\sigma,\alpha}\right)^2
		\nonumber\\
		&\le  \frac{12(N+1)e^{2\rho N^{\mu}}}{\alpha}
		\left(\left\|P \right\|_{a-\delta}^{\sigma,\alpha}\right)^2,
	\end{align}
	where in the last inequality, we use
	$\frac{4(N+1))e^{2\rho N^{\mu}}}{\alpha}
	\left\|P \right\|_{a-\delta}^{\sigma,\alpha}\leq1.$
	Analogously, noting that $\|-\Pi_N P + [P]\|^{\sigma,\alpha}_{a-\delta} \le 2\|P\|^{\sigma,\alpha}_{a-\delta}$, we can obtain:
	
	\begin{equation}\label{est_P+2}
		\begin{split}
			\left\|\sum_{n=2}^{\infty}			\frac{\mathbf{ad}_W^{n-1}}{n!}(-\Pi_N P + [P])	\right\|_{a-\delta}^{\sigma,\alpha}&\le	\left\| P \right\|_{a-\delta}^{\sigma,\alpha}		
			\sum_{n=2}^{\infty}
			\frac{1}{n!} \left( \frac{4(N+1)e^{2\rho N^{\mu}}}{\alpha}
			\left\|P \right\|_{a-\delta}^{\sigma,\alpha}\right)^{n-1}
			\\
			&\le  \frac{12(N+1)e^{2\rho N^{\mu}}}{\alpha}
			\left(\left\|P \right\|_{a-\delta}^{\sigma,\alpha}\right)^2.
		\end{split}
	\end{equation}
	
	By combining \eqref{P1+} with the bounds \eqref{est_P+1} and \eqref{est_P+2}, we derive that
	\begin{align}\label{est_a-b}
		\left\|P^+ \right\|_{a-\delta}^{\sigma,\alpha}
		&\le 	\left\|\Pi_N^{\perp} P \right\|_{a-\delta}^{\sigma,\alpha} + \frac{24(N+1)e^{2\rho N^{\mu}}}{\alpha}
		\left(\left\|P \right\|_{a-b}^{\sigma,\alpha}\right)^2
		\nonumber \\
		&\le e^{-\delta N} \left\|P \right\|_{a}^{\sigma,\alpha}
		+  \frac{24(N+1)e^{2\rho N^{\mu}}}{\alpha}
		\left(\left\|P \right\|_{a}^{\sigma,\alpha}\right)^2.
	\end{align}
\end{proof}

\subsubsection{The iteration process}\

%

Set  $$\epsilon_0:=e^{\frac{C_{\sigma,\tau,a}}{\gamma}}\varepsilon, \quad a_0:=\frac{a}{2},\quad \alpha:=\epsilon_0^{\frac{1}{5}}, \quad c_0:=\sum^{\infty}_{m=1}\frac{1}{m^2}.$$
For $m=1,2,\cdots$, we specify the precise setup of the iteration parameters:
\begin{itemize}
		
	\item$\epsilon_m=\epsilon_{0}^{(\frac{3}{2})^m}$, \quad  the size of the perturbation. 
			
	\item$\delta_{m}=\frac{a}{ 2 c_0 (m+1)^2}$,  \quad the loss of the exponential weight.
	
	\item$a_{m}=a_{m-1}-\delta_{m-1}$, \quad the exponential weight  of the operator at the $m$-th step.
	
	\item$N_m=\delta_m^{-1}(\frac{3}{2})^{m}
	|\ln\epsilon_0|$, \quad  the truncation length of the operator.


\end{itemize}
Then, by the choice of $N_m$, we have
\begin{equation}\label{ }
	e^{-\delta_mN_m}=\epsilon_m
	\quad\text{and}\quad
	a_{\infty}=\lim_{m\rightarrow\infty}a_m=\frac{a_0}{2}.
\end{equation}

\begin{thm}\label{Iter_th}
	(Iteration theorem)
Fix $0<\mu<1$, $\rho>0$, and $\tau,\sigma>1$.
	Then, there exists a sufficiently small constant $\epsilon_0:=\epsilon(\gamma,a,\varepsilon,\mu,\rho,\tau)$ such that if
	\begin{equation}
	\left\|P^0 \right\|_{a_0}^{\sigma,\alpha}\le  \epsilon_0 ,
	\end{equation}
	then for any $m\ge0$, the following statements hold:
	
	\textbf{$($T1$)_m$}: Define the initial nonresonant set $\mathcal{O}_{-1}=\mathrm{DC}_{\gamma,\mu} \cap \mathcal{O}$. For $m\geq 0$, recursively set 
		\begin{equation}\label{set_O^m}
		\mathcal{O}_{m}=\left\lbrace
		\theta\in \mathcal{O}_{m-1}:
		\left|D_{\theta}^m(n)-D_{\theta}^m(n')\right|\geq
		\frac{\alpha}{e^{\rho |n-n'|^{\mu}}}, \quad n\neq n'
		\right\rbrace,
	\end{equation}
	where $D_{\theta}^m$ denotes the diagonal operator constructed in the preceding step. 
	On each of these sets, where $\theta\in \mathcal{O}_{m}$,  there exists a self-adjoint operator $D^m + P^m$. Here, $D_{\theta}^m$ is a real diagonal operator given by:
	\begin{equation}\label{Dm_diag}
		D_{\theta}^m=\mathbf{diag}\big\{ {D}_{\theta}^m(n)
		:n\in\mathbb{Z}^2\big\}\quad\text{with}
		\quad {D}_{\theta}^m(n)=\omega\cdot n+d^m_n(\theta).
	\end{equation}
	Moreover, for $ m\ge 1$, the diagonal entries satisfy the estimate
	\begin{equation}\label{est_D^m}
		 \left| D_{\theta}^{m}(n)-D_{\theta}^{m-1}(n)\right|\le
		 \frac{1}{\ln^{\sigma} \langle n \rangle }
		 \left\| P^{m-1}\right\|_{a_{m-1}}^{\sigma,\alpha}, \quad  \forall\ \theta \in \mathcal{O}_m.
	\end{equation}
	
	\textbf{$($T2$)_m$}: The operator $P^m$ is self-adjoint, belongs to $\mathcal{B}_{a_m}^{\sigma,\alpha}$, and fulfills the norm bound
	\begin{equation}\label{est_P^m}
		\left\|P^m \right\|_{a_m}^{\sigma,\alpha}
		\le \epsilon_m.
	\end{equation}
	
	\textbf{$($T3$)_m$}: There exists a unitary operator $e^{W^m}$ such that the similarity transformation
	\begin{equation*}
		e^{W^m} \left(D^m+P^m \right)  e^{-W^m}
		=D^{m+1}+P^{m+1}
	\end{equation*}
	holds. The generator $W^m$is skew-adjoint, belongs to $\mathcal{B}_{a_m}^{\sigma,\alpha}$, and 
	satisfies
	\begin{equation}\label{est_W^m}
		\left\|W^m \right\|_{a_m}^{\sigma,\alpha}
		\le \epsilon_m^{0.7}.
	\end{equation}
\end{thm}
\begin{proof}
	Let $D^0$ and $P^0$ be defined as
	in Lemma \ref{P_new}.
	Fixing $a_0=\frac{a}{2}$, it follows from Lemma \ref{P_new} that one can choose a sufficiently small $\varepsilon$ such that 
	\begin{equation}
		\|P^0\|^{\sigma,\alpha}_{a_0}\leq e^{\frac{C_{\sigma,\tau}}{\gamma}}\varepsilon \leq \epsilon_0,
		\end{equation}
	and $D^0$ is defined on the set $\mathcal{O}_{-1}$.
Consequently, the statements in \textbf{$($T1$)_0$} and \textbf{$($T2$)_0$} hold true.

Next, assume that for some $m\ge0$, the statements in \textbf{$($T1$)_m$} and \textbf{$($T2$)_m$} hold true.
We  show that \textbf{$($T3$)_m$}, \textbf{$($T1$)_{m+1}$} and \textbf{$($T2$)_{m+1}$} also hold ture.

\textbf{Step 1. Proof of statement \textbf{$($T3$)_{m}$}.} We first construct the unitary operator $e^{W_m}$.
From the definition of $||\cdot||_{a}^{\sigma,\alpha}$ in
\eqref{norm_lip} and the iterative assumptions \eqref{Dm_diag}, \eqref{est_D^m}, for any $\theta_{1}, \theta_{2}\in \mathcal{O}_{m}$, we can obtain the variation of the diagonal frequencies:
	\begin{align}
	\left| \frac{\bigtriangleup_{\theta}d^m_n(\theta) }{\bigtriangleup_{\theta}\theta }\right|
	&\le \sum_{s=1}^{m}\frac{\left| \bigtriangleup_{\theta}\left(D_{\theta}^s(n)- D_{\theta}^{s-1}(n)\right)  \right|} {\left|  \bigtriangleup_{\theta}\theta \right| }
	\nonumber\\
	&\le  \frac{1}{\alpha\ln^{\sigma} \langle n \rangle } \sum_{s=1}^{m}
	\left\| P^{s-1}\right\|_{a_{s-1}}^{\sigma,\alpha}
	\nonumber\\
	&\le  \frac{1}{\alpha\ln^{\sigma} \langle n \rangle }
	 \sum_{s=0}^{m-1}  \epsilon_{s}
	 \le  \frac{ 2\epsilon_{0} }{\alpha}\nonumber\\
	 &\leq 1\nonumber.
	\end{align}	
	This implies that assumption \eqref{non_res2}
	in Lemma \ref{Solu} holds.
	Then, according to Lemma \ref{Solu}, there exists an operator $W^m$ satisfying
	\begin{align}\label{est_Wm}
		\left\|W^m\right\|_{a_m,\sigma}^{\alpha}
		\le&
			\frac{(2N_m+2)e^{2\rho N_m^{\mu} }}{\alpha}
		\left\|P \right\|_{a_m}^{\sigma,\alpha}	\nonumber\\
		\le& \frac{\epsilon^{-0.1}_m}{\alpha}\epsilon_m \le \epsilon^{0.7}_m,
	\end{align}
	where we have used
    \eqref{est_P^m} and the fact that 
	\begin{equation}\label{est_small}
		\begin{split}
		2(N_m+1)e^{2\rho N_m^{\mu} }&\leq \Big(\frac{4c_0(m+1)^2}{a}(\frac{3}{2})^m|\ln \epsilon_0|+2\Big)\exp{\Big(\frac{4 \rho c_0(m+1)^2}{a}(\frac{3}{2})^{\mu m}|\ln \epsilon_0|^{\mu}\Big)}\\
		& \leq \epsilon^{-0.05}_0 \exp{\big(0.05 (\frac{3}{2})^m|\ln \epsilon_0|\big)}\\
		& \leq \epsilon_m^{-0.05} \epsilon_m^{-0.05} \leq\epsilon_m^{-0.1}.
		\end{split}
	\end{equation}
	The second inequality holds	since $0<\mu<1$  and $\epsilon_0=\epsilon_0(a,\rho,\mu,\alpha)$ is chosen to be sufficiently small.

\textbf{Step 2. Proof of statements \textbf{$($T1$)_{m+1}$} and \textbf{$($T2$)_{m+1}$}.} From \eqref{est_P^m} in \textbf{$($T2$)_{m}$}, one can obtain 
\begin{equation}
	\frac{4(N_m+1)e^{2\rho N_m^{\mu} }}{\alpha}
\left\|P \right\|_{a_m}^{\sigma,\alpha}	
\le 2\frac{\epsilon^{-0.1}_m}{\alpha}\epsilon_m \le 2\epsilon^{0.7}_m
<1.
\end{equation}
Hence, assumption \eqref{ass2} in Lemma \ref{P_0} is satisfied. We can obtain
\begin{equation*}
	e^{ W^m} (D^m+P^m) e^{-W^m}=D^{m+1} + P^{m+1},
\end{equation*}
where
\begin{equation*}\label{shift}
	D^{m+1}=D^{m}+[P^m]
\end{equation*}
and
\begin{align}\label{est_P^m+1}
	\left\| P^{m+1} \right\|_{a_{m+1}}^{\sigma,\alpha}
	&\le
	e^{-\delta_m N_m} \left\|P^m \right\|_{a_m}^{\sigma,\alpha}
	+ \frac{24(N_m+1)e^{2\rho N_m^{\mu}}}{\alpha}
 	\left(\left\|P^m \right\|_{a_m}^{\sigma,\alpha}\right)^2.
\end{align}

Also, for any $n\in\mathbb{Z}^2$, it follows from \eqref{shift} that
	\begin{equation*}
	\left| D_{\theta}^{m+1}(n)-D_{\theta}^{m}(n)\right|
	\le \left| P^m_{n,n}  \right|
	\le
	\frac{1}{\ln^{\sigma} \langle n \rangle }
	\left\| P^{m}\right\|_{a_{m}}^{\sigma,\alpha},
\end{equation*}
which implies that
	\begin{equation}\label{est_shift}
	\left|d^{m+1}_n(\theta)-d_n^0(\theta) \right|
	\le
	\frac{1}{\ln^{\sigma} \langle n \rangle }\left(
\sum_{s=0}^{m}	\left\| P^{s}\right\|_{a_{s}}^{\sigma,\alpha}
	\right) 	\le
	\frac{2\epsilon_0}{\ln^{\sigma} \langle n \rangle }.
\end{equation}

Substituting the bound \eqref{est_small} into \eqref{est_P^m+1}, we have
\begin{align*}
	\left\| P^{m+1} \right\|_{a_{m+1}}^{\sigma,\alpha}
	&\le \epsilon_m\ e^{-\delta_m N_m}+\frac{24(N_m+1)e^{2\rho N_m^{\mu}}}{\alpha}\epsilon_m^2\\
	&\le \epsilon_m^{2} + 12\epsilon_{m}^{0.7}\epsilon_m\\
	&\le \epsilon_m^{\frac32}
	= \epsilon_{m+1},
\end{align*}
where we have used the fact that $e^{-\delta_m N_m}$ and the quadratic term are sufficiently small due to \eqref{est_small}.
Thus, the estimates \eqref{est_D^m} and
\eqref{est_P^m} for $m+1$ are verified.
\end{proof}

\begin{thm}\label{conver}Fix the constants $a,\sigma,\alpha, \tau, \mu$ and $\rho$ as in Theorem \ref{Iter_th}.
	Let
	\begin{equation*}
		\left\|P^0 \right\|_{a_0}^{\sigma,\alpha}\le \epsilon_0.
	\end{equation*}
If $\epsilon_0$ is sufficiently small,
then for any $\theta\in \cap_{m=0}^{\infty} \mathcal{O}_{m}$, the sequence of compositions
\begin{equation}\label{sequence}
	\mathcal{W}_m=e^{W^m}\circ \cdots \circ e^{W^0}
\end{equation}
converges in $\left\| \cdot \right\|_{a/2,\sigma}^{\alpha}$
to a unitary operator $\mathcal{W}_{\infty}$ satisfying
\begin{equation}
	\left\| \mathcal{W}_{\infty}-\text{Id} \right\|_{a/2,\sigma}^{\alpha}
	+ \left\| \mathcal{W}^{-1}_{\infty}-\text{Id} \right\|_{a/2,\sigma}^{\alpha}\le \epsilon_0^{0.5}.
\end{equation}
Moreover, there exists a diagonal operator $D^{\infty}$ such that
\begin{equation*}
	\mathcal{W}_{\infty} (D^0+P^0) \mathcal{W}^{-1}_{\infty}=D^{\infty}
\end{equation*}
and
\begin{equation}\label{est_D^inf}
	\left| D_{\theta}^{\infty}(n)-D_{\theta}^{0}(n)\right|\le
	\frac{2\epsilon_{0}}{\ln^{\sigma} \langle n \rangle}.
\end{equation}
\end{thm}

\begin{proof}
	First, we prove that $\{\mathcal{W}_{m}\}_{m\ge 0} $ is a bounded sequence. Note that
	\begin{equation}
		e^{W^m}=\text{Id}+\sum_{n=1}^{\infty}\frac{(W^m)^n}{n!}.
	\end{equation}
	Then by using \eqref{est_W^m}, we have
	\begin{equation}\label{W^m-id}
		\left\| 	e^{W^m}-\text{Id}\right\|_{a/2,\sigma}^{\alpha}
		\le 2	\left\| {W}_{m}\right\|_{a_m,\sigma}^{\alpha}
		\le 2 \epsilon_m^{0.7}.
	\end{equation}
	Moreover, the composition sequence $\mathcal{W}_m$ satisfies
	\begin{equation}\label{W_m-id}
			\mathcal{W}_m=e^{W^m}\circ\mathcal{W}_{m-1}
			=\mathcal{W}_{m-1}+\left( e^{W^m}-\text{Id} \right) \circ	\mathcal{W}_{m-1}
	\end{equation}
	which leads to
	\begin{align}\label{est_W_m}
		\left\| \mathcal{W}_{m}\right\|_{a/2,\sigma}^{\alpha}
		&\le \left\| \mathcal{W}_{m-1}\right\|_{a/2,\sigma}^{\alpha}+
		\left\|  e^{W^m}-\text{Id} \right\|_{a/2,\sigma}^{\alpha}		\left\|\mathcal{W}_{m-1}\right\|_{a/2,\sigma}^{\alpha}
		\nonumber\\
		&\le \left\| \mathcal{W}_{m-1}\right\|_{a/2,\sigma}^{\alpha}
		\left(1+ 2 \epsilon_m^{0.7} \right).
	\end{align}
	Iterating the above inequality, we deduce that for any $m \ge 0$:
		\begin{align}\label{est_W_m1}
		\left\| \mathcal{W}_{m}\right\|_{a/2,\sigma}^{\alpha}
		\le \prod_{s=0}^{m}
		\left(1+ 2\epsilon_s^{0.7} \right)
		\le e,
	\end{align}
	where the last inequality follows from
	\begin{equation*}
	\ln   \prod_{s=0}^{m}
	\left(1+ \epsilon_m^{0.9} \right)
	 = \sum_{s=0}^{m} \ln \left( 1+2\epsilon_s^{0.7}\right)
	  \le   \sum_{s=0}^{m}  2\epsilon_s^{0.7} \le \epsilon_0^{0.5}.
	\end{equation*}

	Next, we show that
	$\{\mathcal{W}_{m}\}_{m\ge 0} $ is a Cauchy sequence with respect to $\left\| \cdot \right\|_{a/2,\sigma}^{\alpha}$.
	
	 By using \eqref{W_m-id}, \eqref{est_W_m} and \eqref{est_W_m1},  one obtains for any $m\ge1$:
	\begin{align*}
		\left\| \mathcal{W}_{m}-\mathcal{W}_{m-1} \right\|_{a/2,\sigma}^{\alpha}
	\le 2\epsilon_m^{0.7}	\left\| \mathcal{W}_{m-1}\right\|_{a/2,\sigma}^{\alpha}
	\le 2e\epsilon_m^{0.7},
	\end{align*}
	and then for any $m\ge0$ and $n\ge1$:
	\begin{align}\label{W_m-n}
		\left\| \mathcal{W}_{m+n} - \mathcal{W}_{m}\right\|_{a/2,\sigma}^{\alpha}
		&\le \sum_{s=m}^{m+n-1}
		\left\| \mathcal{W}_{s+1} - \mathcal{W}_{s}\right\|_{a/2,\sigma}^{\alpha}
		\nonumber\\
		&\le 2e \sum_{s=m+1}^{m+n} \epsilon_{s}^{0.7}
		\le \epsilon_{m+1}^{0.5},
	\end{align}
	which tends to $0$ as $m\to \infty$.
	Thus, the sequence	$\{\mathcal{W}_{m}\}_{m\ge 0} $ converges with respect to $\left\| \cdot \right\|_{a/2,\sigma}^{\alpha}$ to an operator
	$\mathcal{W}_{\infty}$. Furthermore, from \eqref{W^m-id} and \eqref{W_m-n},  $\mathcal{W}_{\infty}$
	 satisfies
	\begin{equation}
		\begin{split}
		\left\| \mathcal{W}_{\infty} - \text{Id}\right\|_{a/2,\sigma}^{\alpha}
		&\le 	\left\| \mathcal{W}_{\infty} - \mathcal{W}_{0} \right\|_{a/2,\sigma}^{\alpha}
		+ 	\left\|  \mathcal{W}_{0}-\text{Id} \right\|_{a/2,\sigma}^{\alpha}\\
	&\le  \epsilon_1^{0.5}+2\epsilon_{0}^{0.7}\\
		&\le \epsilon_{0}^{0.5}.
		\end{split}
 	\end{equation}

Similarly, we can show that the sequence of inverses
\begin{equation*}
	\mathcal{W}^{-1}_m= e^{-W^0} \circ \cdots \circ e^{-W^m}
\end{equation*}
	is also a Cauchy sequence, and $\mathcal{W}^{-1}_m$ converges to $\mathcal{W}^{-1}_{\infty}$ satisfying the corresponding estimate.
In addition, according to \eqref{sequence}, the operator $\mathcal{W}_{m}$ is unitary
since each operator $e^{W^m}$ is unitary.
Therefore, we have
\begin{equation*}
	\mathcal{W}_{\infty} \mathcal{W}^*_{\infty}
	= \lim\limits_{m\to \infty}\mathcal{W}_{m} \mathcal{W}^*_{m}=
	\text{Id}
\end{equation*}
and
\begin{equation*}
	\mathcal{W}^*_{\infty} \mathcal{W}_{\infty}
	= \lim\limits_{m\to \infty}\mathcal{W}^*_{m} \mathcal{W}_{m}=\text{Id},
\end{equation*}
which implies that $\mathcal{W}_{\infty}$ is indeed a unitary operator.
	
We also show that the sequence $\{D_{\theta}^m(n)\}_{m\geq0}$ is Cauchy and converges to $D^{\infty}(n)$, which  satisfies
	\begin{align*}
		\left|  D_{\theta}^{\infty}(n)-D_{\theta}^{0}(n) \right|
		&\le \sum_{m=0}^{\infty} \left|P^m_{n,n} \right|
		\le  \frac{1}{\ln^{\sigma} \langle n \rangle} \sum_{m=0}^{\infty}  \left\| P^{m} \right\|_{a_{m}}^{\sigma,\alpha}	
		\nonumber\\
		&\le \frac{1}{\ln^{\sigma} \langle n \rangle} \sum_{m=0}^{\infty}  \epsilon_{m}\le \frac{2\epsilon_{0}}{\ln^{\sigma} \langle n \rangle},
	\end{align*}
	which implies \eqref{est_D^inf}. This completes the proof.
\end{proof}

\subsubsection{Measure estimates}\

From Theorem \ref{conver}, we can define the diagonal operator $D^{\infty}$:
\begin{equation}
	\begin{split}
D^{\infty}&=\mathbf{diag}\big\{D^{\infty}_n(\theta): n\in \mathbb{Z}^2\big\}
\end{split}
\end{equation}
where 
$$D^{\infty}_n(\theta)=\omega(\theta)\cdot n+d^{\infty}_{n}(\theta),$$
and the nonresonant set 
\begin{equation}\label{seti}
	\widetilde{\mathcal{O}}^{\infty}:=\big\{\theta \in \mathcal{O}_{-1}: |D^{\infty}_n(\theta)-D^{\infty}_{n'}(\theta)|\geq \frac{2\alpha}{e^{\rho|n-n'|^{\mu}}}, \quad n\neq n'\big\}.
\end{equation}
\begin{lem}
	The non-resonant set $\tilde{\mathcal{O}}^\infty$ defined in \eqref{seti} satisfies
	$$\widetilde{\mathcal{O}}^{\infty}\subseteq 
	\bigcap^{\infty}_{s=-1}\mathcal{O}_s .$$
\end{lem}
\begin{proof}From Definition \eqref{seti}, we see $\widetilde{\mathcal{O}}^{\infty}   \subseteq \mathcal{O}_1$. For any $s \geq 0 $, we have
	\begin{equation*}
		\begin{split}
		|D^{\infty}_n(\theta)-D^{s}_n(\theta)|\leq \sum^{\infty}_{i=s}|(P^i)_{n,n}|
		\leq \sum^{\infty}_{i=s} \epsilon_i\frac{1}{\ln^{\sigma} \langle n\rangle}
		\leq 2\epsilon_s.
		\end{split}
		\end{equation*}
For any $\theta \in \widetilde{\mathcal{O}}^{\infty} $, and $|n-n'| \leq N_{s}$,
\begin{equation*}
	\begin{split}
			|D^{s}_n(\theta)-D^{s}_{n'}(\theta)| &\geq |D^{\infty}_n(\theta)-D^{\infty}_{n'}(\theta)|- |D^{s}_n(\theta)-D^{\infty}_{n}(\theta)+D^{s}_{n'}(\theta)-D^{\infty}_{n'}(\theta)|\\
			& \geq \frac{2\alpha}{e^{\rho|n-n'|^{\mu}}}-4\epsilon_s\\
			& \geq  \frac{\alpha}{e^{\rho|n-n'|^{\mu}}}.
			\end{split}
	\end{equation*}		
	The last inequality holds since $4\epsilon_se^{\rho N_s^{\mu}} \leq 4\epsilon_s\epsilon^{-0.1}_s \leq \alpha$.
	\end{proof}
Since we only introduce angles as parameters to avoid resonance, we also need the following lemma for measure estimates.
\begin{lem}\label{Determinant}
	 Let $u^{(0)}, \cdots, u^{(r-1)}$ be $r$ independent vectors in $\mathbb{R}^r$
	 and $w \in \mathbb{R}^r$ be an arbitrary vector. Then there exists $i \in \{0, \cdots, r-1\}$, such that
	\[ \left|  w\cdot u^{(i)} \right|  \geq
	\frac{\|w\|_{\ell_1} |\det(u^{(0)},\cdots, u^{(r-1)} )| }{r \left\| u^{(0)},\cdots, u^{(r-1)} \right\|^{r-1}}, \]
	where 	
	$$\left\| u^{(0)},\cdots, u^{(r-1)} \right\|:=\max_{0\le i\le r-1}	\left\|u^{(i)} \right\|_{\ell_1}  $$
	and
	$\det(u^{(0)},\cdots, u^{(r-1)})$ is the determinant of the matrix formed by the components of the vectors $u^{(0)},\cdots,u^{(r-1)}$.

\end{lem}
\begin{proof}
The proof can be found in Lemma A.5 in \cite{SW24}.
\end{proof}
%

	\begin{lem}\label{Mes2}
		 Suppose that $g(m)$ is $r$ times differentiable on an interval $J \subset \mathbb{R}$. For any $\gamma > 0$, let $J_\gamma := \{m \in J : |g(m)| < \gamma\}$. If there exists a constant $d>0$ such that $|g^{(r)}(m)| \geq d$ for all $m \in J$, then the Lebesgue measure of $J_\gamma$
		 satisfies $|J_\gamma| \leq M \gamma^{1/r}$, where $M := 2(2 + 3 + \cdots + r + d^{-1})$. Here $|\cdot|$ denotes the Lebesgue measure of a set.
	\end{lem}
	
	\begin{proof} The proof can be found in Lemma A.8 in \cite{SW24}. \end{proof}

Under these settings, we have the following lemma.
\begin{lem}\label{Mes3}
	Let $\tau>1$.
	For any  $\gamma>0$, one has 
$$meas(\mathcal{O} \setminus \mathcal{O}_{-1})=O(\gamma).$$
\end{lem}
\begin{proof}
	Given a vector $\ell\in \mathbb{Z}^2$, define the non-Diophantine set
		\begin{equation}
		\mathcal{R}_{\ell}:=\left\lbrace
		\theta\in \mathcal{O}:
	\left| \omega(\theta)\cdot\ell \right|  < \frac{\gamma}{|\ell|^{ \tau }}
		\right\rbrace
	\end{equation}
	and
		\begin{equation}
		\mathcal{O} \setminus \mathcal{O}_{-1}:= \bigcup_{0<|\ell|<\infty} 	\mathcal{R}_{\ell}.
	\end{equation}
	
	Consider the vector $\omega(\theta) = (\cos \theta, \sin \theta)$ and its derivative $\omega'(\theta) = (-\sin \theta, \cos \theta)$. Their determinant satisfies
	\begin{equation}
		\det(\omega(\theta), \omega'(\theta)) = \begin{vmatrix}
			\cos\theta & 	\sin\theta \\
			-\sin\theta & 	\cos\theta
		\end{vmatrix}= 1.
	\end{equation}
	In view of Lemma \ref{Determinant}, for any $\theta \in \mathcal{O}$, there exists an index $i \in \{0, 1\}$ such that
	\begin{equation}\label{derivate}
		| \omega^{(i)}(\theta)\cdot \ell| \geq \frac{|\ell| \det(\omega(\theta), \omega'(\theta))}{2 \max\{|\omega|, |\omega'|\}} \geq \frac{|\ell|}{4},
	\end{equation}
	where $\omega^{(0)}=\omega(\theta)$ and $\omega^{(1)}=\omega'(\theta)$.
	
	Applying Lemma \ref{Mes2} to the function $g(\theta) =   \omega(\theta)\cdot\ell$, where the $i$-th derivative is bounded below by $|\ell|/4$, we deduce
	\begin{equation}\label{mes_R_ell}
		\text{meas}(\mathcal{R}_{\ell}) \le  \frac{10\gamma}{|\ell|^{\tau}}.
	\end{equation}
	
	Summing over $\ell \in \mathbb{Z}^2 \setminus \{0\}$, and noting that $\tau > 1$ ensures the convergence of the series, we obtain
	\begin{align}
		\text{meas}(\mathcal{O} \setminus \mathcal{O}_{-1}) \le \sum_{\ell \neq 0} \text{meas}(\mathcal{R}_{\ell}) \le \sum_{\ell \neq 0} \frac{10\gamma}{|\ell|^{\tau}} \le C(\tau) \gamma,
	\end{align}
	where $C(\tau)$ is a constant depending on $\tau$.
\end{proof}

\begin{lem}
Let $\gamma>\alpha>0$ be a sufficiently small parameter, and let $\tau>1$, $0 < \mu < 1$, and $\rho>0$ be constants. Suppose that  $\sigma$ satisfies the scaling relation
	\begin{equation}\label{para}
	\sigma>\frac{\tau}{\mu}>1.
	\end{equation}
	Then, one has 
	$$meas(\mathcal{O}_{-1} \setminus \widetilde{\mathcal{O}}^{\infty} )=O(\alpha).$$
\end{lem}

\begin{proof}	
	For any $n,n'\in \mathbb{Z}^2$ with $0<|n-n'|<\infty$, define the resonant set
	\begin{equation}
		\mathcal{R}_{n,n'}:=\left\lbrace
		\theta\in \mathcal{O}_{-1}:
		\left|  D_{n}^{\infty}(\theta)-D_{n'}^{\infty}(\theta) \right|  < \frac{2\alpha}{\exp(\rho|n-n'|^{ \mu })}
		\right\rbrace
	\end{equation}
	and  
	\begin{equation}
		\mathcal{O}_{-1} \setminus \widetilde{\mathcal{O}}^{\infty}:= \bigcup_{n,n'\in \mathbb{Z}^2} 	\mathcal{R}_{n,n'}.
	\end{equation}
	We prove the measure estimate by induction and a partition of the index space.
	
	\textbf{Case 1: Large index.}
%
%
Considering the condition
\begin{equation}\label{min_n_0}
	\min_{n_0 \in \{n,n'\}} \langle n_0 \rangle \ge \exp\left(\frac{\rho}{2}|n-n'|^{\tau/\sigma}\right),
\end{equation}
then by using the inductive assumption and the smallness of the perturbation, 
we obtain for any $\theta\in \mathcal{O}_{-1} $:
\begin{align}
	\left| D_{n}^{\infty}(\theta)-D_{n'}^{\infty}(\theta) \right|
	&\ge \left| \omega(\theta) \cdot (n-n')\right| - 2\max_{n_0 \in \{n,n'\}} |d^{\infty}_{n_0}(\theta)| \nonumber\\
	&\ge \frac{\gamma}{|n-n'|^{\tau}} - \max_{n_0 \in \{n,n'\}} \frac{4\epsilon_0}{\ln^{\sigma} \langle n_0 \rangle} \nonumber\\
	&\ge \frac{\gamma}{|n-n'|^{\tau}} - \frac{4\epsilon_0}{(\frac{\rho}{2})^\sigma |n-n'|^\tau}\\
	\nonumber &\ge \frac{\gamma}{2|n-n'|^{\tau}}\\
	& \ge \frac{2\alpha}{\exp(\rho|n-n'|^{ \mu })},
\end{align}
where we used the fact that $\gamma$ is sufficiently large relative to $\alpha$. This implies that under condition \eqref{min_n_0}, $\mathcal{R}_{n,n'} = \emptyset$.


\textbf{Case 2: Small Index.}
Therefore, we only need to focus on the case
\begin{equation}\label{min_n_0'}
	\min_{n_0 \in \{n,n'\}} \langle n_0 \rangle < \exp\left(\frac{\rho}{2}|n-n'|^{\tau/\sigma}\right).
\end{equation}

In view of the transversality condition \eqref{derivate}, for any $\theta \in \mathcal{O}_{-1}$, there exists an index $i \in \{0, 1\}$ such that:
\begin{align}
     \nonumber	\left| \left( D_{n}^{\infty}(\theta)-D_{n'}^{\infty}(\theta)\right)^{(i)}   \right|
	&\geq \left| \omega^{(i)}\cdot \left(n-n' \right)\right | - 2	\max_{n_0\in\{n,n'\}} \left| \left( d^{\infty}_{n_0}(\theta)\right)^{(i)} \right| \\
		&\geq\frac{1}{4}-\frac{4\epsilon_{0}}{\alpha}
		\geq \frac{1}{8}.
		\end{align}
Then, following the proof of \eqref{mes_R_ell}, one has
\begin{equation*}
		\text{meas}(\mathcal{R}_{n,n'})\le \frac{20 \alpha}{\exp(\rho|n-n'|^{ \mu })}.
\end{equation*}
Summing over all possible pairs $(n, n')$ satisfying \eqref{min_n_0'}, and letting $k=n-n'$:
\begin{align*}
	\text{meas}(\mathcal{O}_{-1}\setminus \widetilde{\mathcal{O}}^{\infty}) &\le \sum_{k \in \mathbb{Z}^2 \setminus\{0\}} \sum_{\substack{n_0\in\mathbb{Z}^2
			\atop{\langle n_0 \rangle \le e^{\rho|k|^{\tau/\sigma}}}}} \text{meas} (\mathcal{R}_{n_0, n_0+k})\\
	&\le \sum_{k \in \mathbb{Z}^2 \setminus \{0\}} 4 \exp\left(\rho|k|^{\tau/\sigma}\right) \cdot \frac{20\alpha}{\exp\left(\rho|k|^{\mu}\right)} \\
	&\le 80\alpha \sum_{k \in \mathbb{Z}^2 \setminus \{0\}} \exp\left( \rho|k|^{\tau/\sigma} - \rho|k|^{\mu} \right).
\end{align*}
Since  $\mu > \tau/\sigma$, the exponent $|k|^{\tau/\sigma} - |k|^{\mu} \leq -\frac{1}{2}|k|^{\mu}$  as $|k| \to \infty$. Thus the series converges, and we conclude:
\begin{equation*}
	\text{meas}(\mathcal{O}_{-1}\setminus \widetilde{\mathcal{O}}^{\infty}) \le C(\rho,\tau,\sigma,\mu) \alpha.
\end{equation*}

This completes the proof.	
\end{proof}

\subsection{Proof of Theorem \ref{thm2}}
\begin{proof}
	From the Lemma \ref{P_new} and Theorem \ref{conver}, for any $\theta \in \widetilde{\mathcal{O}}_{\infty}$, we can define the unitary operator 
	\begin{equation}
	\mathrm{Q}=\mathcal{W}_{\infty}\cdot e^{W}.
	\end{equation}
	One sees 
	\begin{equation}
	\begin{split}
	\|\mathrm{Q}\|_{\frac{a}{4}}&\leq \|\mathcal{W}_{\infty}\|_{\frac{a}{4}}\|e^{W}\|_{\frac{a}{4}}\\
	&\leq e^{\frac{C_{a,\sigma,\tau}}{\gamma}}(1+\epsilon_0^{0.5})\\
	&\leq C(\gamma,\tau,a,\sigma, \varepsilon).
	\end{split}
	\end{equation}
		Similarly, one can get $\|\mathrm{Q}^{-1}\|_{\frac{a}{4}} \leq C(\gamma,\tau,a,\sigma, \varepsilon)$.
Thus, one knows 
\begin{equation}
\mathrm{Q}\mathcal{H}_{\theta,\varepsilon} \mathrm{Q}^{-1}=D^{\infty},
\end{equation}
where 
$$D^{\infty}=\mathbf{diag}\big\{D^{\infty}_n(\theta): n\in \mathbb{Z}^2\big\}.$$
Letting $\psi_{n}=\mathrm{Q}^{-1}\delta_{n}$, one sees that
\begin{equation}
\mathcal{H}_{\theta,\varepsilon}\psi_{n}=D^{\infty}_n(\theta) \psi_n,
\end{equation}
and \begin{equation}
	\psi_{n}(m)=\sum_{n' \in \mathbb{Z}^2}(\mathrm{Q}^{-1}_{n,n'})\delta_{n'}(m)=\mathrm{Q}^{-1}_{n,m}.
	\end{equation}
	Finally, one gets 
	\begin{equation}
		\begin{split}
		|\psi_{n}(m)| &\leq |\mathrm{Q}^{-1}_{n,m}|\leq \|\mathrm{Q}^{-1}\|_{\frac{a}{4}}e^{-\frac{a}{4}|m-n|}\\
		&\leq C(\gamma,\sigma, \tau,a,\varepsilon)e^{-\frac{a}{4}|m-n|}.
		\end{split}
	\end{equation}
	\end{proof}

\section{Appendix}

\begin{lem}\label{lem26.5.8}
Let $A_1=\Delta+\omega_1 n$ with $\omega_1\neq0$ be defined on the domain 
$$\mathcal{D}_0=\{\psi\in\ell^2(\mathbb{Z}):\psi(n)=0 \ \mbox{for all but finitely many $n$}\}.$$
Then $A_1$ is essentially self-adjoint on $\mathcal{D}_0$.	
\end{lem}

\begin{proof}
	For any $\phi,\psi\in\mathcal{D}_0$, a direct computation shows
$$\langle \phi,A_1\psi\rangle=\langle A_1\phi,\psi\rangle,$$
because $\Delta$ is symmetric and $\omega_1n$ is a real multiplication operator. Hence $A_1$ is symmetric on $\mathcal{D}_0$. 

Since $A_1$ is a Jacobi (tridiagonal) operator with real coefficients, its adjoint $A_1^*$
acts as the same formal difference expression on the maximal domain
$$\mathcal{D}(A_1^*)=\{\psi\in\ell^2(\mathbb{Z}):\ A_1\psi\in\ell^2(\mathbb{Z})\}.$$
Thus we need to show: if $\psi\in\ell^2(\mathbb{Z})$ satisfies
 $$(\Delta\psi)(n)+\omega_1n\psi(n)=\pm\mathbf{i}\psi(n),\quad
 \forall n\in\mathbb{Z}$$
 then $\psi\equiv0$.
	
	Recall $(\Delta\psi)(n)=\psi(n-1)+2\psi(n)+\psi(n+1)$. Substituting, we get
	$$\psi(n-1)+2\psi(n)+\psi(n+1)+\omega_1n\psi(n)=\pm \mathbf{i}\psi(n).$$
	Collecting terms:
	$$\psi(n-1)+2\psi(n)+\psi(n+1)+(\omega_1n\mp\mathbf{i})\psi(n)=0.$$
	Thus we get the standard three-term recurrence:
	\begin{equation}\label{26.5.8.2}
		\psi(n-1)+\psi(n+1)+(\omega_1n+2\mp\mathbf{i})\psi(n)=0.\end{equation}
		
Since $\omega_1\neq0$, without loss of generality, we may assume that $\omega_1>0.$ Consider the half-line $n\geq1$. Equation \eqref{26.5.8.2} is a second-order linear difference equation. For a Jacobi operator with coefficients
	$$a_n\equiv1,\quad b_n=\omega_1n+2\mp\mathbf{i},$$
	the Carleman condition (sufficient for the limit-point case at $+\infty$) is
	$$\sum_{n=1}^{\infty}\frac{1}{\sqrt{|b_n|}}=+\infty.$$
	Here $|b_n|\sim \omega_1 |n|$ for large $|n|$. Since
	$$\sum_{n=1}^{\infty}\frac{1}{\sqrt{\omega_1n}}=\frac{1}{\sqrt{\omega_1}}\sum_{n=1}^{\infty}\frac{1}{\sqrt{n}}=+\infty,$$
	the Carleman condition holds. Hence at most one linearly independent solution of \eqref{26.5.8.2} is square-summable as 
	$n\rightarrow\infty$.
	
	Equation \eqref{26.5.8.2} is a homogeneous linear difference equation of order $2$. Its solution space on $\mathbb{Z}$ is two-dimensional. If there were a nonzero $\ell^2(\mathbb{Z})$-solution, it would have to be square-summable both as $n\rightarrow+\infty$ and as $n\rightarrow-\infty$.
	
	By the Carleman condition for the right half-line ($n\rightarrow+\infty$), the space of $\ell^2$-solutions at $+\infty$ has dimension at most $1$. By symmetry (replace $n$ by $-n$ and note that $\omega_1n$ changes sign), the same Carleman condition applies to the left half-line $n\rightarrow-\infty$
	after the change of variable $m=-n$. Explicitly, for large negative $n$, $|b_n|\sim \omega_1|n|$ again, so $\sum_{n=-\infty}^{-1}\frac{1}{\sqrt{|b_n|}}=+\infty$, ensuring the limit-point case at $-\infty$ as well.
	
	Thus the space of global $\ell^2$-solutions of \eqref{26.5.8.2} can have dimension at most $1$. More precisely, since the operator is in the limit-point case at both $\pm\infty$, any nontrivial solution that is $\ell^2$ at $+\infty$ cannot be $\ell^2$ at $-\infty$ (and vice versa) unless the two one-dimensional subspaces of square-summable solutions happen to coincide. However, a nonzero global $\ell^2(\mathbb{Z})$-solution would imply that $\pm\mathbf{i}$ are eigenvalues of the self-adjoint extensions, which is impossible since the operator is symmetric and $\pm\mathbf{i}\notin \mathbb{R}$. 
	
	Therefore the equations $A_1^*\psi=\pm \mathbf{i}\psi$ have only the trivial solution $\psi\equiv0$ in $\ell^2(\mathbb{Z})$. Hence, the deficiency indices are $(0,0)$, which implies that $A_1$ is essentially self-adjoint on $\mathcal{D}_0$.
	\end{proof}

\section*{Declarations}

\textbf{Funding} Meina Gao is supported by National Natural Science Foundation of
China (Grant No.12571195). Siming Li is supported by National Natural Science Foundation of China (Grant No.125B2011). Yingte Sun was supported by National Natural Science Foundation of China (Grant No.12101542). \\

\textbf{Conflict of interest}  On behalf of all authors, the corresponding author states that there is no conflict of interest.\\

\textbf{Availability of data }  No data was used for the research described in the article.

\bibliographystyle{alpha} 
\bibliography{ref1}

\end{document}